\documentclass[11pt,a4paper]{amsart}
\usepackage[toc,page]{appendix}
\usepackage{a4}
\usepackage[active]{srcltx}
\usepackage{amsmath,bbm}
\usepackage{amssymb}
\usepackage{amsfonts}
\usepackage{MnSymbol}
\usepackage{mathrsfs} 
\usepackage{graphicx}
\usepackage{color}
\usepackage{soul}\setstcolor{red} 
\usepackage[colorlinks=true,urlcolor=blue,
citecolor=red,linkcolor=blue,linktocpage,pdfpagelabels,bookmarksnumbered,bookmarksopen]{hyperref}

\def\a{\alpha}
\def\b{\beta}
\def\d{\delta}

\def\g{\gamma}
\def\k{\kappa}

\def\e{\varepsilon}

\def\Om{\Omega}

\newcommand{\cA}{{\mathcal A}}
\newcommand{\cB}{{\mathcal B}}

\newcommand{\cF}{{\mathcal F}}

\newcommand{\cK}{{\mathcal K}}
\newcommand{\cJ}{{\mathcal J}}

\newcommand{\cN}{{\mathcal N}}
\newcommand{\cC}{{\mathcal C}}

\newcommand{\p}{\partial}

\newcommand{\R}{{\mathbb R}}

\newcommand{\ov}{\overline}

\newcommand{\mres}{\mathbin{\vrule height 1.6ex depth 0pt width
0.13ex\vrule height 0.13ex depth 0pt width 1.3ex}}

\newcommand{\sL}{\mathscr{L}}

\newcommand{\sM}{{\mathscr M}}

\def\dd{\,{\rm d}}

\newcommand{\one}{\mathbbm{1}}

\newcommand{\ds}{\displaystyle}

\newcommand{\diver}{{\rm{div}}}
\newcommand{\dist}{{\rm{dist}}}

\newcommand{\iin}{{\rm{in}}}
\newcommand{\oon}{{\rm{on}}}

\newcommand{\spt}{{\rm{spt}}}

\makeatletter
\newcommand{\mathleft}{\@fleqntrue\@mathmargin0pt}
\newcommand{\mathcenter}{\@fleqnfalse}
\makeatother

\newtheorem{remark}{\textbf{Remark}}[section]
\newtheorem{theorem}{\textbf{Theorem}}[section]
\newtheorem{lemma}[theorem]{\textbf{Lemma}}
\newtheorem{corollary}[theorem]{\textbf{Corollary}}
\newtheorem{proposition}[theorem]{\textbf{Proposition}}

\numberwithin{equation}{section}

\title[Variational formulation of some stationary mean field games]{On the variational formulation of some stationary second order mean field games systems}
\author[A.R. M\'esz\'aros]{Alp\'ar Rich\'ard M\'esz\'aros}  
\address{Department of Mathematics, UCLA, 520 Portola Plaza, Los Angeles, CA 90095, USA}
\email{alpar@math.ucla.edu} 

\author[F. J. Silva]{Francisco J. Silva} 
\address{Institut de recherche XLIM-DMI, Universit\'e de Limoges, 87060 Limoges, France}
\email{francisco.silva@unilim.fr}

\date{\today}

\begin{document}
\maketitle
\begin{abstract}  We consider the variational approach to prove the existence of solutions of second order stationary Mean Field Games on a bounded domain $\Om\subseteq \R^{d}$, with Neumann boundary conditions, and with and without density constraints. We consider Hamiltonians which growth as $|\cdot|^{q'}$, where $q'=q/(q-1)$ and $q>d$. Despite this restriction, our approach allows us to prove  the existence of solutions in the case of rather general coupling terms. When   density constraints are taken into account,  our results improve those in \cite{MesSil}.  Furthermore, our approach can be used to obtain solutions of systems with multiple populations. 
\end{abstract}
\section{Introduction}
In this article we consider the stationary Mean Field Game (MFG) system 
\begin{equation}\label{eq:MFG_1}
\left\{
\begin{array}{rll} 
- \Delta u+ H(\cdot,\nabla u) +\lambda &= f(x,m) &  \iin\ \Om, \\[6pt]
\nabla u\cdot n &= 0  & \oon\ \partial \Om, \\[6pt]
-\Delta m - \diver\left(m \nabla_\xi H(\cdot,\nabla u)\right)  &= 0,  & \iin\ \Omega, \\[6pt]
(\nabla m+m \nabla_\xi H(\cdot,\nabla u))\cdot n &= 0,   & \oon\ \partial\Omega,\\[6pt]
\ds\int_\Omega m\dd x=1,&     m(x)> 0 &  \iin\ \ov{\Omega},
\end{array}
\right.\tag{MFG$_{1}$}
\end{equation}
introduced by Lasry and Lions in \cite{LasLio06i,LasLio07}, which models   the equilibrium configuration of an ergodic stochastic symmetric differential game  with a continuum of small,  indistinguishable  players (see e.g. \cite{MR3127148}). In the above system, $\Omega \subseteq \R^{d}$ is a bounded domain with a smooth boundary, $n$ is the outward normal to $\p\Omega$, $f: \Omega \times [0,+\infty[\to \R$ is the so-called {\it local coupling function},  $\Omega \times \R^d\ni (x,\xi) \mapsto H(x,\xi) \in \R$ is the {\it Hamiltonian}, and the variables $m: \Omega \to   [0,+\infty) $, $u: \Omega \to \R$ and $\lambda \in \R$ represent the stationary equilibrium configuration of the players, the equilibrium cost of a typical player, and the ergodic constant, respectively. We have taken Neumann boundary conditions but our results admit natural versions for the more standard case of periodic boundary conditions.

 An interesting feature of the solutions of \eqref{eq:MFG_1} is their connection with the long time behaviour of the solutions of time-dependent MFGs. We refer the reader to \cite{CarLasLioPor2,CarLasLioPor} for some results justifying rigorously this relation in some special cases.  The numerical resolution of stationary second order MFGs has been studied in \cite{AchdouCapuzzo10,Cacace_camilli_16,Almulla2016,briceno_kalise_silva}. 

Existence and uniqueness results for system \eqref{eq:MFG_1} have been investigated by several   researchers  using Partial Differential Equation (PDE) techniques, starting with the first papers \cite{LasLio06i,LasLio07} in the framework of weak solutions. The reader is referred to \cite{MR3333058,ferreiragomes15,MR3530210,Cir} for other subsequent results on the existence of weak solutions. In addition, under  different assumptions on the coupling function $f$ and the growth of $H$, the existence, and uniqueness, of smooth solutions have been analysed in \cite{MR3160525,MR3415027,pimentel2015regularity}. See also \cite{MR2928381}, where several a priori estimates for smooth solutions of stationary second order MFGs are established.  

In this article  we focus our attention on the proof of the existence of weak solutions of \eqref{eq:MFG_1} by variational techniques. Indeed, as pointed out already in \cite{LasLio07},  system   \eqref{eq:MFG_1} can be seen, formally, as the first order optimality condition of an associated variational problem, involving a PDE   constraint for the variable $m$. It turns out that $u$ and $\lambda$ in  \eqref{eq:MFG_1} correspond to the Lagrange multipliers associated to the PDE constraint for $m$ and the condition $\int_{\Om}m(x) \dd x=1$, respectively. Given $q>d\geq 2$, where $d$ is the space dimension, and setting $q':=q/(q-1)$, we prove the above assertion for Hamiltonians $H$ growing as $|\cdot|^{q'}$. Even if this growth condition is restrictive, but crucial for our arguments, the main interest of this variational technique is that it allows to prove the existence of weak solutions of \eqref{eq:MFG_1} for a rather general class of coupling functions $f$ in a straightforward manner. Indeed, as we will show in Section \ref{variational_problems}, $f$ does not need to be monotone (see also \cite{Cir, CirGomPimSan} for some recent results in this direction) and, moreover, we can prove the existence of solutions of   variations of system \eqref{eq:MFG_1} involving couplings  which can also depend  on the distributional derivatives of $m$. As a matter of fact, our results are valid, for terms in the r.h.s. of the first equation in \eqref{eq:MFG_1} which can be identified with the derivative of a   function $\cF: W^{1,q}(\Om) \to \R$  which is G\^ateaux dfferentiable  and weakly  lower-semicontinuous. 

Our approach follows closely the one in \cite{MesSil}, which   considers  in addition a density constraint in order to model strong congestion effects (see \cite{filippo,MR3556062}). In that article, the existence of solutions $(m,u,p,\mu,\lambda)\in W^{1,q}(\Om)\times W^{1,q'}(\Om)\times  \sM(\ov{\Omega})\times \sM(\ov{\Omega} ) \times \R$, where $  \sM(\ov{\Omega}) $ is the set of Radon measures on $\ov{\Omega}$, to the system 
\begin{equation}\label{eq:MFG_2_previous_article}
\left\{
\begin{array}{rll} 
- \Delta u+ \frac{1}{q'}|\nabla u|^{q'}+\mu- p+ \lambda &= f(x,m) &  \iin\ \Om, \\[6pt]
\nabla u\cdot n &= 0  & \oon\ \partial \Om, \\[6pt]
-\Delta m - \diver\left( |\nabla u|^{q'-2} \nabla u m\right)  &= 0,  & \iin\ \Omega, \\[6pt]
 \nabla m\cdot n &= 0,   & \oon\ \partial\Omega,\\[6pt]
\ds\int_\Omega m\dd x=1,&   0 \leq m(x)\le 1 \hspace{0.5cm}  &  \iin\ \ov{\Omega},\\[8pt]
 \mbox{{\rm spt}}(p) \subseteq \{ x\in \ov{\Om} \; ; \; m(x)=  1\},  &  p \geq 0 &  \iin\ \ov{\Omega},\\[8pt]
  \mbox{{\rm spt}}(\mu) \subseteq \{ x\in \ov{\Om} \; ; \; m(x)= 0\},  &  \mu \geq 0 &  \iin\ \ov{\Omega},
\end{array}
\right. 
\end{equation}
is established if $f(x,\cdot)$ is non-decreasing.  If $q>d$, the  result is proved with  a variational approach. On the other hand,  when $1 <q \leq d$ a penalization argument allows to prove the existence also in this case. In the present article,  when $q>d$, we improve the results in \cite{MesSil}, and we show the existence of solutions of   
\begin{equation}\label{eq:MFG_2}
\left\{
\begin{array}{rll} 
- \Delta u+ H(\cdot,\nabla u)- p+ \lambda &= f(x,m) &  \iin\ \Om, \\[6pt]
\nabla u\cdot n &= 0  & \oon\ \partial \Om, \\[6pt]
-\Delta m - \diver\left(m \nabla_\xi H(\cdot,\nabla u)\right)  &= 0,  & \iin\ \Omega, \\[6pt]
(\nabla m+m \nabla_\xi H(\cdot,\nabla u))\cdot n &= 0,   & \oon\ \partial\Omega,\\[6pt]
\ds\int_\Omega m\dd x=1,&   0 < m(x)\le \k(x) \hspace{0.5cm}  &  \iin\ \ov{\Omega},\\[8pt]
 \mbox{{\rm spt}}(p) \subseteq \{ x\in \ov{\Om} \; ; \; m(x)=  \kappa(x)\},  &  p \geq 0 &  \iin\ \ov{\Omega}.
\end{array}
\right.\tag{MFG$_{2}$}
\end{equation}
Note that in \eqref{eq:MFG_2} a more general Hamiltonian is considered and the density constraint $m\leq 1$ is replaced by $m \leq \kappa$, where $\kappa \in W^{1,q}(\Omega)$. Most importantly, $f$ does not need to be monotone and, using the Harnack's inequality proved in \cite{MR0369884} (see also \cite{MR3443169}) for elliptic equations in divergence form, we show that the density $m$ is strictly positive, which implies that $\mu$ in \eqref{eq:MFG_2_previous_article} is identically zero. Using the existence of solutions of the variational problem associated to \eqref{eq:MFG_2}, which can be proved easily, we prove that the variational problem associated to  \eqref{eq:MFG_1} admits at least one solution. This crucial fact  is the key to show  the existence of solutions of \eqref{eq:MFG_1}.

The paper is organized as follows. In Section \ref{prelim} we begin with some preliminaries which allow us to characterize the subdifferential of the cost functionals appearing in the optimization problems associated to \eqref{eq:MFG_1}-\eqref{eq:MFG_2}. This analysis extends the one in \cite[Section 2]{MesSil}. Section \ref{variational_problems} is the core of the article. We prove the existence of solutions of the variational problems associated to \eqref{eq:MFG_1}-\eqref{eq:MFG_2} and we establish the corresponding optimality conditions, which provide the existence of solutions of  \eqref{eq:MFG_1}-\eqref{eq:MFG_2}. We present a detailed discussion concerning the generality of the coupling term, which, as we have explained before, is the main feature of this approach. We also prove, by a bootstrapping argument, additional regularity for the weak solutions.   In Section \ref{multipopulation_section}, we present some simple applications of our results to the study of multi-populations MFG systems (see e.g.  \cite{MR3333058,MR3530210}). Finally, in the appendix, we prove the strict positivity of the densities   $m$  appearing in  \eqref{eq:MFG_1}-\eqref{eq:MFG_2} as a consequence of the Harnack's inequality in \cite{MR0369884} and the assumed regularity of the boundary $\partial \Omega$.

\section{Preliminary results}\label{prelim} 
In the entire article, we will assume that $\Omega \subseteq \R^{d}$ ($d\ge 2$) is a non-empty, bounded open set with a $C^{1,1}$ boundary $\partial \Omega$. This regularity assumption is equivalent to a uniform interior and exterior ball condition (see for instance \cite[Theorems 1.8-1.9]{Dalphin}) and allows us to use the classical Sobolev inequalities. 
The vector $n$ will denote the outward normal to $\p\Omega$.  Given $r \in [1,+\infty]$ and $\ell \in \mathbb{N}$ we will denote by $\|\cdot \|_{r}$ and $\| \cdot\|_{\ell,r}$ the standard norms in $L^{r}(\Om)$ (or in $L^{r}(\Om)^d$)  and $W^{\ell,r}(\Om)$, respectively.

Let $q>d \; \geq 2$. Our aim in this section is to provide  a characterization of the subdifferential of the convex functional
$\cB_q: W^{1,q}(\Omega)\times L^{q}(\Om)^{d}\to \ov\R \; := \R \cup \{+\infty\}$,  defined as 
\begin{equation}\label{deflqcal} 
\cB_q(m,w):= \int_{\Omega} b_q(x,m(x), w(x))\dd x,
\end{equation}  
where  
$b_q:\Om\times\R\times\R^d\to\ov\R$ is given by 
\begin{equation}\label{deflq}b_q(x,m, w):= \left\{\begin{array}{ll} 
\ds m H^*(x,-w/m) & \mbox{if } \; m>0, \\[6pt]	
		0 & \mbox{if } (m,w)=(0,0), \\[6pt]
		+\infty & \mbox{otherwise}.\end{array}\right.\end{equation}	
In view of the assumptions below, the function $\cB_q$ is well defined and convex (see Remark \ref{Bqwelldefined} and Theorem \ref{thm:subdiff}). It will appear in the cost functional of an optimization problem whose first order optimality condition  has the form of an MFG system. 
In \eqref{deflq},  for every $x\in \Omega$ the function $H^{\ast}(x,\cdot): \R^d \to \R$  is  the Legendre-Fenchel transform of $H(x,\cdot)$, where $H: \Omega \times \R^d\to \R$ is a continuous function, that we will call the {\it Hamiltonian},  which is assumed to be  strictly convex and differentiable in its second variable and satisfies a polynomial growth condition in terms of $q':= q/(q-1)$:
there exist $C_1$, $C_2>0$ such that
\begin{equation}\label{hyp:1}
\frac{1}{q'C_1} |\xi|^{q'}-C_2\le H(x,\xi)\le \frac{C_1}{q'}|\xi|^{q'}+C_2, \forall\, x\in\Om,\;\;\; \xi\in\R^d.
\end{equation} 
Using the definition of $H^{\ast}(x,\cdot)$, an easy computation shows that \eqref{hyp:1} implies   
\begin{equation}\label{growthhstar}
	\frac{C_1^{1-q}}{q}|\eta|^{q}-C_2\le H^{\ast}(x,\eta)\le \frac{C_1^{q-1}}{q}|\eta|^{q}+C_2, \hspace{0.3cm} \forall \, x\in \Om, \; \; \; \eta \in \R^d.
\end{equation} 
Note that since  $H(x,\cdot)$ is strictly convex and differentiable, we have that $H^{\ast}(x,\cdot)$ is also strictly convex and differentiable (see e.g. \cite[Theorem 26.3]{MR0274683}). We denote their gradients  by $\nabla_\xi H(x,\cdot)$ and $\nabla_\eta H^{\ast}(x,\cdot)$, respectively. Moreover, classical results in convex analysis show that for any $x\in\Om$ 
\begin{equation}\label{invertibilidadnablaH} \nabla_\xi H(x,\cdot)^{-1}(\eta)= \nabla_\eta H^{\ast}(x,\eta) \hspace{0.3cm} \forall \; \eta \in \R^d.
\end{equation}
We now prove an elementary result which will be useful later. In the remainder of this article will denote by $C>0$  a generic constant which can change from line to line.
\begin{lemma}\label{basiclemmaonH}  Under \eqref{hyp:1} we have that $H^{\ast}(\cdot,\cdot)$, $\nabla_\xi H(\cdot,\cdot)$ and $\nabla_\eta H^{\ast}(\cdot, \cdot)$ are continuous. Moreover,  if $\beta \in L^{q'}(\Om)^d$ we have that $\nabla_\xi H(\cdot, \beta(\cdot)) \in L^{q}(\Om)^d$. Analogously, if $v\in L^{q}(\Om)^d$ then  $\nabla_\eta H^{\ast}(\cdot, v(\cdot)) \in L^{q'}(\Om)^d$.
\end{lemma}	
\begin{proof} Let $(x_n,\eta_n)$ and $(x,\eta)$ in $\ov\Om\times\R^d$ be such that $(x_n,\eta_n)\to (x,\eta)$ as $n\to \infty$.  Set $\xi_n:= \nabla_\eta H^{\ast}(x_n,\eta_n)$. By definition of $H^{\ast}$
$$
H^{\ast}(x_n, \eta_n)= \eta_n \cdot \xi_n - H(x_n, \xi_n)\geq \eta_n\cdot \xi -H(x_n, \xi) \hspace{0.3cm} \forall \; \; \xi \in \R^d.
$$
Since $\eta_n$ is  convergent, hence bounded, the first inequality in \eqref{hyp:1} shows that $\xi_n$ is bounded.  Let $\bar{\xi}$ be a limit point of $\xi_n$. Then, using the continuity of $H$ and passing to the limit, up to some subsequence, we get that 
$$
 \eta \cdot \bar{\xi} - H(x, \bar{\xi})\geq \eta \cdot \xi -H(x, \xi) \hspace{0.3cm} \forall \; \; \xi \in \R^d, 
$$
which shows that $\bar{\xi}= \nabla H^{\ast}(x,\eta)$ and any limit point of $H^{\ast}(x_n, \eta_n)$ is equal to $H^{\ast}(x, \eta)$.  The continuity of $\nabla H(\cdot,\cdot)$ follows by the symmetric argument.  Finally, if  $\beta \in L^{q'}(\Om)^d$, setting $x \in \Om \mapsto \eta(x):= \nabla H(x, \beta(x))$, which is a measurable function since $\nabla H(\cdot, \cdot)$ is continuous, we get by convexity that 
$$
H^{\ast}(x,\eta(x)) \leq \beta(x)\cdot \eta(x) +H^{\ast}(x,0),
$$
and so, by \eqref{growthhstar}, we obtain the existence of $C>0$ such that $|\eta(x)|^{q} \leq C  (\beta(x)\cdot\eta(x) +1)$. Using Young's inequality, we get the existence of $C>0$ such that $|\eta(x)|^{q} \leq C (|\beta(x)|^{q'} +1)$ and thus, integrating in $\Omega$, we obtain that $\eta \in L^{q}(\Om)$.  The last assertion follows from an analogous argument.
\end{proof}	
Regarding the dependence of $H$ on the space variable $x$, we will assume that there exists a modulus of continuity which is uniform w.r.t. the second variable, i.e.  $\exists \; \omega:  [0,+\infty) \to  [0,+\infty) $ such that $\omega(0)=0$, $\omega$ is continuous, non-decreasing and 
\begin{equation}\label{hyp:2}
|H(x,\xi)-H(y,\xi)|\le \omega(|x-y|)(|\xi|^{q'} + 1), \hspace{0.3cm} \forall \; x,  y \in \Om, \; \; \xi \in \R^d. 
\end{equation}
Using that $\ov{\Om}$ is a compact set, a natural example of a Hamiltonian $H$ satisfying \eqref{hyp:1} and \eqref{hyp:2} is given by $H(x,\xi):= b(x)|\xi|^{q'} + c(x)$ where $b, c \in C(\ov{\Om})$ and $b>0$. 



Following the analysis in \cite[Lemma 2.1 and Theorem 2.2]{MesSil}, presented in a more particular setting, we shall characterize the subdifferential of $\cB_q$, defined in \eqref{deflqcal}. 
Recall that given a normed space $(X,\|\cdot\|)$ and a l.s.c. convex proper function $g: X \to \ov \R$, the subdifferential $\partial g(x)$ of $f$ at the point $x$, consist in the set of all $x^{*} \in X^{\ast}$ such that 
$$ g(x) + \langle x^{*}, y-x \rangle_{X^{*},X} \leq g(y) \hspace{0.4cm} \forall \; y\in X.$$
For the sake of completeness, in order to identify $\partial \cB_q$, we  first state some simple properties of the function $b_{q}(x, \cdot, \cdot)$. Given $x\in\Om$  consider the set 
\begin{equation}\label{Aqofx}A_{q'}(x):=\{(\a,\b)\in\R\times\R^d\; : \; \a+H(x,-\b)\le 0\}.\end{equation}
Since $H$ is continuous and convex w.r.t. its second variable, we have that $A_{q'}(x)$ is closed and convex  for any $x\in\Om$. Given a subset $D$ of an euclidean space, we denote by $\chi_D$ its characteristic function (in the sense of convex analysis), i.e. $\chi_D(y)=0$ if $y\in D$ and $\chi_D(y)=+\infty$ otherwise. 
\begin{lemma}\label{lemaelemental} For all $x\in \Omega$, the function $b_q(x,\cdot,\cdot)$ is convex, proper and l.s.c. Its Legendre-Fenchel conjugate and its subdifferential  are given by
\begin{equation}\label{conjugatesubdiffellq} b_q^{\ast}(x,\cdot,\cdot)=\chi_{A_{q'}(x)}(\cdot,\cdot), \hspace{0.5cm} \partial_{(m,w)} b_q(x,m,w)= \left\{ \begin{array}{ll} (-H(x,-\beta_x), \beta_x)  & \mbox{{\rm if} } \; m>0, \\[4pt]
							(\alpha,\beta)  \in A_{q'}(x) & \mbox{{\rm if} } \; (m,w)=(0,0), \\[4pt]
							\emptyset & \mbox{{\rm otherwise},}\end{array} \right.\end{equation}
where, if $m>0$,  $\beta_x:=-\nabla H^{\ast}(x,-w/m)$.
\end{lemma}
\begin{proof} Using \eqref{deflq}, it is straightforward to check that for all $x\in \ov{\Om}$ we have 
\begin{equation}\label{ellqislsc}b_q(x,m,w)= \sup_{\beta \in \R^{d}} \left\{ \beta \cdot w - m H(x,-\beta) \right\},
\end{equation}
and so $b_q(x,\cdot,\cdot)$ is convex, proper and l.s.c.  Using that $b_q(x,\cdot,\cdot)$ is proper, we have that $\partial_{(m,w)} b_q(x,m,w)=\emptyset$ if $m<0$ or $m=0$ and $w\neq 0$.  On the other hand, from \eqref{ellqislsc},  if $m\geq 0$, for any $x\in\ov\Om$  we have the identity 
\begin{equation}\label{twoequalities} b_q(x,m,w)= \sup\left\{ \alpha m+\beta \cdot w \; : \; (\alpha,\beta) \in A_{q'}(x) \right\},  \end{equation}
which is checked to hold also when $m<0$. Thus, since $A_{q'}(x)$ is closed and convex, $b_q^{*}(x,\cdot,\cdot)= \chi_{A_{q'}(x)}(\cdot,\cdot)$. This expression directly yields that $\partial_{(m,w)} b_q (x,0,0)= A_{q'}(x)$. If $m>0$, then $b_q(x,\cdot,\cdot)$ is differentiable and, by a simple computation, we get the expression of its gradient with respect to $(m,w)$. 
\end{proof}
\begin{remark}\label{Bqwelldefined} Notice that the equality in \eqref{twoequalities} shows that $(x,m,w) \mapsto b_q(x,m,w)$ is lower-semicontinuous and so, by \cite[Example 14.31]{MR1491362}, we have that $b_q$ is a normal integrand. This shows that  $ \Omega\ni x \mapsto b_{q}(x,m(x),w(x))$ is a measurable function if $m$ and $w$ are measurable  {\rm(}see \cite[Proposition 14.28]{MR1491362}{\rm)}. In particular, the  functional $\cB_{q}$ is well defined.
\end{remark}
Let us define\begin{equation}\label{expressionovA}
\cA_{q'}:=\left\{(a,b)\in L^\infty(\Om)\times L^\infty(\Om)^d\; : \; (a(x),b(x))\in A_{q'}(x)\ {\rm{for\ a.e.}}\ x\in\Om\right\}, 
\end{equation}
 and denote by $\sM_+(\ov{\Om})$ and $\sM_{-}(\ov{\Om})$ the subsets of  nonnegative and  nonpositive finite Radon measures of  $\sM(\ov{\Omega})$, respectively.  For a set $D$, we denote by $\one_D$ its indicator function, i.e. $\one_D(y)=1$ if $y\in D$ and $\one_D(y)=0$ otherwise. 
\begin{lemma}\label{lem:subdiff}
The closure of $\cA_{q'}$ in $(W^{1,q}(\Om))^*\times L^{q'}(\Om)^d$ is given by
\begin{equation}\label{closure}
\ov{\cA_{q'}}:=\left\{(\a,\b)\in\sM(\ov\Om)\times L^{q'}(\Om)^d\; : \; \a+H(\cdot,-\b) \in  \sM_{-}(\ov\Om) \right\}, 
\end{equation}
or equivalently, 
\begin{equation}\label{closure1}\ov{\cA_{q'}}:=\left\{(\a,\b)\in\sM(\ov\Om)\times L^{q'}(\Om)^d\; : \; \a^{\rm{ac}}+H(\cdot,-\b) \leq 0, \; \; \mbox{{\rm a.e. in} $\Om$ {\rm and } } \; \alpha^{\rm{s}} \in \sM_{-}(\ov{\Om}) \right\}, 
\end{equation}
where $ \dd \alpha= \alpha^{\rm{ac}} \dd x + \dd  \alpha^{\rm{s}}$ is the Lebesgue decomposition of the measure $\alpha$ w.r.t. the Lebesgue measure restricted to $\Om$. 
\end{lemma}
\begin{proof}  Let us take $(\a_n,\b_n)\in\cA_{q'}$ converging  to $(\a,\b)$ in $(W^{1,q}(\Om))^*\times L^{q'}(\Om).$ By  definition of $\cA_{q'}$ one has 
$$\int_\Om\a_n(x)\phi(x)\dd x\le-\int_{\Om}H(x,-\b_n(x))\phi(x)\dd x,\;\;\forall\phi\in W^{1,q}(\Om), \; \; \phi \ge 0.$$
Since $\b_n\to\b$ in $L^{q'}(\Om)^d$, up to some subsequence $\b_n(x)\to\b(x)$ a.e. in $\Om$. By \eqref{hyp:1} we can use  Fatou's lemma to obtain that
$$\langle\a,\phi\rangle_{(W^{1,q})*,W^{1,q}}\le -\int_{\Om}H(x,-\b(x))\phi(x)\dd x,\;\;\forall\phi\in W^{1,q}(\Om),\; \; \phi \ge 0.$$
Using  \eqref{hyp:1} again, we obtain that $\a-C_2$ defines a nonpositive distribution, hence by \cite[Th\'eor\`eme V]{Schw66} $\a$ can be identified with an element of $\sM(\ov\Om).$

Conversely let us take $(\a,\b)$ belonging to the right-hand-side (r.h.s.) of \eqref{closure}, or equivalently to the r.h.s. of \eqref{closure1}.
Analogously to \cite{MesSil}, we construct different approximations for $\a^{\rm{ac}}$ and $\b$ on the one hand and for $\a^{\rm{s}}$ on the other hand. For $R>0$ and $x\in \R^d$ we set $B_{R}(x)=\{ y\in \R^d \; : \; |y-x|<R\}$. Consider a mollifier $\eta: \R^d \to \R$ satisfying that $\eta \in C_c^{\infty}(\R^d)$, $\eta\geq 0$, $\ds\int_{\R^d}\eta(x)\,\dd x=1,$ $\spt(\eta)\subseteq B_1(0)$ and $\eta(x)=\eta(-x)$ for all $x\in \R^d$. Now, for $\e>0$ set
$$\Omega_\e:=\{x\in\Omega:\ \dist(x,\partial\Omega)>\e\}, \hspace{0.3cm} \eta_{\e}(x):= \frac{1}{\e^{d}} \eta(x/\e),$$ and for all $x\in \Omega$ and $i=1,..., d$, let us define
$$\tilde\alpha_{\e}(x):= \int_{\Omega} \eta_{\e}(x-y)  \alpha^{{\rm ac}}(y)\dd y \one_{\Om_\e}(x) , \hspace{0.5cm} \tilde\b_{\e}^{i}(x):=\int_{\Omega} \eta_{\e}(x-y)\beta^{i}(y) \dd y \one_{\Om_\e}(x). $$
As $\e\downarrow 0$, we have that  $\tilde\a_\e \to \a^{\rm{ac}}$ in $L^{1}(\Om)$,  $ \tilde\b_\e \to \b$ in $L^{q'}(\Omega)^{d}$ and $(\tilde\a_\e,\tilde\b_\e)\in L^{\infty}(\Om)\times  L^{\infty}(\Om)^d$. Multiplying the   inequality in \eqref{closure1}  by $\eta_\e$, integrating and using  Jensen's inequality yield
\begin{equation}\label{cut}
 \tilde\a_\e(x)+H(x,-\tilde\b_\e(x))\le \tilde\d_\e(x), \hspace{0.3cm} \forall x\in \Om,
\end{equation}
where 
$$\tilde\d_\e(x):= \int_{B_\e(x)}\left[H(x,-\b(y))- H(y,-\b(y))\right] \eta_\e (x-y)\dd y\one_{\Om_\e}(x).$$
Note that  \eqref{hyp:2} implies that 
\begin{equation}\label{inquality_proof}
|\tilde\d_\e(x)| \leq \omega(\e)  \int_{B_\e(x)} \left(1+ |\b(y)|^{q'}\right) \eta_\e (x-y)\dd y\one_{\Om_\e}(x) \leq \frac{\omega(\e)}{\e^{d}}\|\eta\|_{\infty}\|1+|\beta|^{q'} \|_{1} ,
\end{equation}
and so $\tilde{\d}_\e\in L^{\infty}(\Om)$. Using that $ g_\e(\cdot):=\int_{\Om} \left(1+ |\b(y)|^{q'}\right) \eta_\e (\cdot-y)\dd y\one_{\Om_\e}(\cdot)$ converges in $L^1(\Om)$ to $1+ |\b(\cdot)|^{q'}\in L^1(\Om)$, extracting a subsequence, the first inequality in \eqref{inquality_proof} implies that $\tilde\d_\e(\cdot) \to 0$ a.e. in $\Omega$. Since for $\e \in (0,1)$ we have that $|\tilde\d_\e| \leq \omega(1) g_\e$, we get from \cite[Chapter 1.3, Theorem 4]{evans-gariepy} that $\tilde\d_\e\to 0$ in $L^1(\Om)$.
%
%
This implies that $(\hat{\a}_{\e}, \hat{\b}_{\e}):=( \tilde\a_\e-\tilde\d_\e,  \tilde\b_\e)\in \mathcal{A}_{q'}$ and $(\hat{\a}_{\e}, \hat{\b}_{\e})\to (\alpha^{{\rm ac}}, \beta)$ in $L^1(\Om)\times L^{q'}(\Om)^{d}$.

Now, in order to approximate the singular part $\alpha^{{\rm s}}$,  for $x\in\ov\Om$ and $\e>0$ let us define $\ds\rho^x_\e:= \left(\one_{B_\e(x)\cap\ov\Om}\right) / |B_\e(x)\cap\ov\Om|$ and 
$$\hat{\a}^{{\rm s}}_\e(y):=\int_{\ov\Om}\rho_\e^x(y)\,\dd\a^{\rm{s}}(x) \hspace{0.3cm} \forall \; y\in \ov{\Om},$$
which is a non-positive  function. Arguing exactly as in the proof of \cite[Lemma 2.1]{MesSil} we get that the  uniform interior ball assumption on the boundary $\partial \Om$ implies that $\hat{\a}^{{\rm s}}_\e \in L^{\infty}(\Om)$. Using that  $\rho^x_\e \to \d_x$  in $\sM(\overline{\Omega})$, as $\e \downarrow 0$,  it is straightforward to show that $\hat{\a}^{{\rm s}}_\e \to \a^{\rm{s}}$ in  $\sM(\overline{\Omega})$. Therefore, the sequence $(\hat{\a}^{{\rm s}}_\e+ \hat{\a}_{\e}, \hat{\b}_{\e})\to (\alpha, \beta)$ in $\sM(\Om)\times L^{q'}(\Om)^{d}$ and $(\hat{\a}^{{\rm s}}_\e+ \hat{\a}_{\e}, \hat{\b}_{\e}) \in \cA_{q'}$ for all  $\e>0$.  Using that $q>d$, and so $W^{1,q}(\Omega) \hookrightarrow C(\ov{\Om})$ by the Sobolev embedding,  we have that $(\hat{\a}^{{\rm s}}_\e+ \hat{\a}_{\e}, \hat{\b}_{\e})\to (\alpha, \beta)$ weakly in   $(W^{1,q}(\Om))^{\ast}\times L^{q'}(\Om)^{d}$. Since $\cA_{q}$ is convex, its closure w.r.t. the weak and strong topologies coincide. The result follows. 
\end{proof}

For a given representative of $m\in W^{1,q}(\Omega)$  in $C(\ov\Om)$, we denote $\{m=0\}:=\left\{x\in\ov\Om:m(x)=0\right\}$ and $\{m>0\}:=\left\{x\in\ov\Om:m(x)>0\right\}$.  

\begin{theorem}\label{thm:subdiff}  The following assertions hold true:\smallskip\\
{\rm(i)} The functional $\cB_q$,  defined in \eqref{deflqcal}, is convex,  l.s.c. and $\cB_r^*(\a,\b)=\chi_{\ov{\cA_{r'}}}(\a,\b)$ for all $(\a,\b)\in (W^{1,q}(\Om))^*\times L^{q'}(\Om)^d$. \smallskip\\
{\rm(ii)} Let $(m,w)\in W^{1,q}(\Omega)\times L^{q}(\Omega)^d$   and  suppose that $\cB_q(m,w)<\infty$.     Then,  if  $v:= (w/m)\one_{\{m>0\}}\notin L^{q}(\Omega)^d$ we have that $\partial\cB_q(m,w)=\emptyset$. 
Otherwise,  $\cB_q$ is subdifferentiable at $(m,w)$ and \small
\begin{equation}\label{subdifferential}
\partial\mathcal{B}_q(m,w)=\left\{(\a,\b)\in \overline{\cA_{q'}} \; : \; \a\mres \{m>0\}=-H(\cdot,\nabla H^*(\cdot,-v)) \; \; \; {\rm and } \; \;  \;  \b\mres \{m>0\}=-\nabla H^*(\cdot,-v)\right\}.
\end{equation} \normalsize
In particular, the singular part of $\alpha$ in \eqref{subdifferential}, w.r.t. to the Lebesgue measure,  is concentrated in $\{m=0\}$. 
\end{theorem}
\begin{proof} Since the arguments are  similar to  those in the proof of \cite[Theorem 2.2]{MesSil}, we only sketch  the main ideas.  First, truncating the sets $A_{q'}(x)$, defined in \eqref{Aqofx}, by setting for $k\in \mathbb{N}$
$$ A_{q',k}(x):= \left\{ (a,b)\in  A_{q'}(x) \; ; \; a \geq -k, \; \; \mbox{max}_{i=1,\hdots, d} |b_{i}| \leq k \right\}, $$
using  Lemma \ref{lemaelemental} and  the monotone convergence theorem, we have that 
$$
\mathcal{B}_{q}(m,w)= \int_{\Om} \sup_{(a,b) \in A_{q'}(x) } \left\{ a m(x) + b\cdot w(x) \right\} \dd x =\lim_{k\to \infty} \int_{\Om} \sup_{(a,b) \in A_{q',k}(x) } \left\{ a m(x) + b\cdot w(x) \right\} \dd x.$$
Characterizing the point-wise optimizers $(\a(x),\b(x))\in A_{q',k}(x)$  in the last expression, it easy to see  that $\Om\ni x \mapsto (\a(x),\b(x))$ is a measurable function, which  by definition belongs to $L^{\infty}(\Om) \times L^{\infty}(\Om)^d$.  Using this fact and the monotone convergence theorem once again, we find that 
$$\begin{array}{rcl}
\mathcal{B}_{q}(m,w)&=& \ds\sup_{(\a, \b) \in \mathcal{A}_{q'}} \int_{\Om}\left\{ \a(x) m(x) + \b(x)\cdot w(x) \right\} \dd x \\[8pt]
\; &=& \ds\sup_{(\a, \b) \in \ov{\mathcal{A}_{q'}}} \int_{\ov{\Om}}   m(x)  \dd \a(x) + \int_{\Om}\b(x)\cdot w(x)  \dd x.\end{array}
$$
  Since $\ov{\mathcal{A}_{q'}}$ is closed and convex in $(W^{1,q'}(\Om))^{\ast}\times (L^{q'}(\Om))^d$, assertion {\rm (i)} follows.  Using that $\mathcal{B}_{q}(m,w) <+\infty$ implies that $w(x)=0$ a.e. where $m(x)=0$,  the definition of $\ov{\mathcal{A}_{q'}}$ implies  that  
\begin{equation}\label{severalsup}
\begin{array}{l}
\ds\sup_{(\a, \b) \in \ov{\mathcal{A}_{q'}}} \int_{\ov{\Om}}   m(x)  \dd \a(x) + \int_{\Om}\b(x)\cdot w(x)  \dd x\\[8pt]
= \ds\sup_{\b \in L^{q'}(\Om)^d}  \int_{\Om} \left[- H(x,-\b(x))m(x)+   \b(x)\cdot w(x)\right] \dd x, \\[8pt]
= \ds\sup_{\b \in L^{q'}(\Om)^d}  \int_{\Om} \left[- H(x,-\b(x))+   \b(x)\cdot v(x)\right]m(x) \dd x.
\end{array}
\end{equation}
An element $\b   \in L^{q'}(\Om)^d$  maximizes  the above expression  iff for a.e. $x \in \{m>0\}$  we have that $ v(x)=-\nabla H(x, -\b(x))$. Since $\beta \in L^{q'}(\Omega)^d$,  if $v\notin L^{q}(\Om)^d$  Lemma \ref{basiclemmaonH} implies that  the previous relation cannot be satisfied and so $\partial \mathcal{B}_{q}(m,w)=\emptyset$.  On the other hand, if $v\in L^{q}(\Om)^d$ then 
any  $\b \in L^{q'}(\Om)^d$ satisfying  a.e. in $\{m>0\}$ that  $\b(x)= -\nabla H^{\ast}(x,-v(x))$  optimizes the last expression in \eqref{severalsup}. Therefore, using the definition of $\mathcal{A}_{q'}$ we readily get that  if $(\alpha, \b) \in \partial \mathcal{B}_{q}(m,w)$, then 
$$ \beta(x)= -\nabla H^{\ast}(x,-v(x)) \; \; \mbox{and } \; \a\mres \{m>0\}=-H(\nabla H^*(x.-v)) \; \; \; \mbox{a.e. in $\{m>0\}$}.$$
The result follows. 
\end{proof}	
\begin{remark} A generalization of the previous result to the case when $1 < q\leq d$ could be interesting by extending      the techniques in \cite{Bre}. However, since our results in the next section are intrinsically related to the assumption   $q>d$, we have preferred to provide  a direct and self-contained  proof in this case. 
\end{remark}
\section{The variational  problems}\label{variational_problems}
 Let us fix $q>d.$ In order to define the variational problems we  are interested in, we   introduce first the data and our assumptions.   Let  
$$
W_{+}^{1,q}(\Om) := \{ m \in W^{1,q}(\Om) \; ; \; m(x) \geq 0 \; \; \forall \; x\in \Om\}, 
$$ 
and $\cF:~W^{1,q}(\Om)\to~\R $ be such that 
\begin{equation}\label{assumption_F}
\begin{array}{c}
\cF \mbox{  is weakly lower  semicontinuous,   G\^ateaux-differentiable in  $W_{+}^{1,q}(\Om)$ }  \\[4pt]
\mbox{and satisfies that } \; \; \forall \, R>0 \; ; \exists \; C_R \in \R \; \mbox{such that if } m\in W^{1,q}(\Om) \\[4pt]
 \mbox{ and  }  \; \;  0\leq m(x) \leq R \; \; \forall \; x\in \Om,  \; \; \; \mbox{then } \;  \cF(m) \geq C_R. 
\end{array}
\end{equation}  
Given $w \in L^{q}(\Omega)^{d}$  let us  consider the following elliptic PDE, with Neumann boundary conditions, 
\begin{equation}\label{principaleq} 
\left\{
\begin{array}{rcl}
-\Delta m + \diver(w)&=&0,  \hspace{0.4cm} \iin\ \Omega,\\[4pt]
(\nabla m - w)\cdot n &=&0, \hspace{0.4cm}  \oon\ \partial\Omega.
\end{array}
\right.
\end{equation}
We say that $m\in W^{1,q}(\Omega)$ is a weak solution of  \eqref{principaleq} if
\begin{equation}\label{principalecweakform}
\int_\Omega \nabla m(x)\cdot\nabla\varphi(x)\dd x =\int_\Omega w(x)\cdot\nabla\varphi(x)\dd x \hspace{0.5cm} \forall \; \phi \in C^{1}(\overline{\Omega}).
\end{equation}
Let us use the notation $Y:= \{ \ell \in W^{1,q'}(\Om) \; ; \; \int_{\Om} \ell(x) \dd x= 0\}$. Since, $\mbox{div}(w) \in Y^{\ast}$, the results in \cite[Section 7.1]{GiaMar} and \cite[Appendix]{MesSil}  imply the existence of a unique $m\in W^{1,q}(\Om)$ such that $m$ is a weak solution of \eqref{principalecweakform} and $\int_{\Om}m(x) \dd x= 1$. Moreover, there exists $C>0$ such that 
\begin{equation} 
\|\nabla m\|_{q} \leq  C \|w \|_{q},
\end{equation}
and so, by the Poincar\'e-Wirtinger inequality, there exists $C>0$ such that 
\begin{equation}\label{estimacion_w1q}
\|  m\|_{1,q} \leq  C \left(\|w \|_{q} +1\right). 
\end{equation}
Note that \eqref{principalecweakform} can be written as $Am+Bw=0$, with  $A: W^{1,q}(\Omega) \to Y^{\ast}$ and $B: L^{q}(\Omega)^d \to Y^{\ast}$ being linear bounded operators defined as
$$\langle Am, \phi  \rangle_{Y^{\ast}, Y}:= \int_{\Omega} \nabla m \cdot \nabla \phi \dd x, \hspace{0.4cm} \langle  Bw, \phi \rangle_{Y^{\ast},Y}:= -\int_{\Omega}w\cdot \nabla \phi \dd x \hspace{0.5cm} \forall  \; \phi \in Y,$$ 
where $\langle \cdot, \cdot \rangle_{Y^{\ast},Y}$ denotes the   duality product between $Y^{\ast}$ and $Y$. 
Now, let us define $\cJ: W^{1,q}(\Omega) \times L^{q}(\Omega)^{d}\to  \ov{\R}$ and   $G: W^{1,q}(\Omega) \times L^{q}(\Omega)^{d} \to  Y^{\ast} \times \R$ as 
$$\cJ(m,w):= \cB_{q}(m,w) + \cF(m), \hspace{0.4cm}G:=(G_1, G_2):= \left(Am +Bw, \int_{\Om}m(x) \dd x - 1\right).$$
The first variational problem we consider is 
\begin{equation}\label{prob:P1}
\inf_{m \in W^{1,q}(\Om), \; w\in L^{q}(\Om)^d} \; \cJ(m,w) \; \; {\rm{such\ that}} \; \; G(m,w)= 0. \;  \tag{P$_1$}
\end{equation}
In the second variational problem we impose a density constraint: let 
$\kappa \in W^{1,q}(\Om)$ be  such that 
\begin{equation}\label{assumptions_on_kappa}
\underline{\kappa}:= \mbox{min}_{x\in \ov{\Om}} \kappa(x)>0 \hspace{0.3cm} \mbox{and } \int_{\Om} \kappa(x) \dd x >1.
\end{equation}
Given a representative of $\kappa$, still denoted by $\kappa$, we define the set 
$$\cC:=\{ m\in C(\ov{\Om}) \; ; \; m (x) \leq \kappa (x) \; \; \forall \; x \in \ov{\Om}\}.$$
We consider the problem 
\begin{equation}\label{prob:P2}
\inf_{m \in W^{1,q}(\Om), \; w\in L^{q}(\Om)^d} \; \cJ(m,w) \; \; {\rm{such\ that}} \; \; G(m,w)= 0, \; \; m\in \cC.  \tag{P$_2$}
\end{equation}
Note that $q>d$ and the Sobolev embeddings imply that $W^{1,q}(\Omega) \hookrightarrow C(\ov{\Om})$  and so the constraint $m\in \cC$ in \eqref{prob:P2} is well-defined. 
%
%
%
%
%
%
\begin{remark}
 Theorem \ref{thm:subdiff}{\rm(i)} and \eqref{assumption_F} imply that the cost functional $\cJ$ in \eqref{prob:P1} and \eqref{prob:P2} is weakly lower semicontinuous. On the other hand, it is not necessarily convex. 
\end{remark}
\subsection{Existence of solutions of the variational problems}
In this subsection we prove that both problems \eqref{prob:P1} and \eqref{prob:P2} admit at least one solution.  The proof of existence of solutions of problem \eqref{prob:P2} follows the same lines than the proof of  \cite[Theorem 3.1]{MesSil}. The proof of existence of solutions for problem \eqref{prob:P1} introduces an artificial density constraint and  uses the existence of solutions of \eqref{prob:P2}.  Given a Lebesgue measurable set $A\subseteq \R^d$ we denote by $|A|$ its Lebesgue measure. 
\begin{theorem}\label{existence0}  Assume that \eqref{assumption_F} holds. Then,   problem  \eqref{prob:P2} has  at least one solution $(m,w)$. If in addition, $\cF$ is bounded from below in $W_{+}^{1,q}(\Om)$,  by a constant $C_\cF\in\R$, then problem \eqref{prob:P1} also admits at least one solution  $(m,w)$. Moreover, in the latter case,  

\begin{equation}\label{estimates_solution_optimization_problem}
\begin{array}{l}
\|m\|_{\infty} \leq  \max\left\{2c_0 c_1, (2c_0c_1)^{q}q C_1^{q-1}\left(\cF(1/|\Om|)+2C_2- C_{\cF}\right)\right\}, \\[6pt]
\mbox{{\rm and}} \quad \; \; \| w\|_{q}^q\leq   q C_1^{q-1}\left(\cF(1/|\Om|)+ 2C_2- C_{\cF} \right)\|m\|_{\infty}^{q-1},
\end{array}
\end{equation}
where $C_1$ and $C_2$ satisfy \eqref{hyp:1} and $c_0$ and $c_1$ depend only on the geometry of $\Om$.
\end{theorem}

\begin{proof} We first prove the assertion for problem \eqref{prob:P2}, where the density constraint allows to obtain directly some bounds on any minimizing sequence.  Define 
\begin{equation}
\hat{\kappa}:= \frac{\kappa}{\|\kappa\|_{1}} \hspace{0.4cm} \mbox{and so } \; \;  0 <\hat{\kappa}(x)<\kappa(x)  \; \; \;  \forall \; x\in \ov{\Om} \; \;  \; \mbox{and } \; \int_{\Om} \hat{\kappa}(x)\dd x=1. 
\end{equation}
By the surjectivity result in  \cite[Lemma A.1]{MesSil}, there exists $\hat{w} \in L^{q}(\Om)^d$ such that $A\hat{\kappa} + B\hat{w}=0$.  By \eqref{growthhstar}, we have that 
$$
\cB_q(\hat{\kappa}, \hat{w}) \leq  \frac{C_1^{q-1}}{q} \int_{\Om} \frac{|\hat{w}(x)|^{q}}{\underline{\kappa}^{q-1}} \dd x + C_2 +\cF(\hat{\kappa})< +\infty,
$$
which implies that the infimum in \eqref{prob:P2} is not $+\infty$. Now, let $(m_n, w_n)$ be a minimizing sequence and set $\ov{\kappa}:= \max_{x\in \ov{\Om}} \kappa(x)$. The previous discussion implies the existence of $C>0$ such that $\cJ(m_n, w_n) \leq C$ for all $n \in \mathbb{N}$. In particular, by \eqref{deflq},    $w_n=0$ a.e. on the set $\{m_n=0\}$. Since $(m_n, w_n)$ is feasible, we get that $(1/m_n^{q-1})\one_{\{m_n>0\}}\geq \ov{\kappa}^{1-q}$ and so  \eqref{growthhstar} and \eqref{assumption_F}, with $R=\ov{\kappa}$, imply that
\begin{equation}\label{estimation_with_the_cost}
\int_{\Omega} |w_{n}|^{q} \dd x= \int_{\{m_n>0\}} |w_{n}|^{q} \dd x \leq q(C_1 \ov{\kappa})^{q-1} \left(C+C_2- C_{\ov{\kappa}} \right).
\end{equation}
Therefore, the sequence $w_n$ is bounded in $L^{q}(\Om)^{d}$ and so, by  \eqref{estimacion_w1q}, the sequence $m_n$ is bounded in $W^{1,q}(\Omega)$. Therefore, extracting a subsequence, we obtain the existence of $(m,w)$ such that $m_n$ converges weakly to $m$ in $W^{1,q}(\Om)$ and $w_n$ converges weakly to $w$ in $L^{q}(\Om)^{d}$. By passing to the weak limit, we get that $G(m,w)=0$ and $m\in \cC$.   Finally,  using that $\cJ$ is weakly lower semicontinuous we get that $(m,w)$ solves \eqref{prob:P2}.

Now, let us prove existence for \eqref{prob:P1} under the additional assumption on  the boundedness  from below of $\cF$ in $W^{1,q}_{+}(\Om)$.  Let $\gamma> 1/|\Om|$. Since $(\hat{m}, \hat{w}):= (1/|\Omega|, 0)$ is feasible for problem \eqref{prob:P2}, with $\kappa(x) \equiv \gamma$,  we have that \eqref{prob:P2} admits at least one solution. We will  show that any such solution $(m_\g, w_\g)$ satisfies that $\|m_\g\|_{\infty}\leq \bar{\kappa}$ for some constant $\bar{\kappa}>0$ which is independent of $\gamma$. This will prove the result since any solution $(m_{\bar{\kappa}},w_{\bar{\kappa}})$ of $(P_{2})$ with $\kappa(x)\equiv \bar{\kappa}$ solves \eqref{prob:P1}. Indeed, if there is a feasible $(m,w)$ for problem \eqref{prob:P1} such that $\cJ(m,w)<\cJ (m_{\bar{\kappa}},w_{\bar{\kappa}})$ then since there exists $\kappa'>0$ such that $\|m\|_{\infty} \leq \kappa'$ (because $m\in W^{1,q}(\Om)$) we have that $\cJ(m_{\kappa'},w_{\kappa'})\leq \cJ(m,w)$, where $(m_{\kappa'},w_{\kappa'})$ is a solution of $(P_{2})$ with $\kappa\equiv \kappa'$,  and $m_{\kappa'}\leq \bar{\kappa}$ which contradicts the optimality of $(m_{\bar{\kappa}},w_{\bar{\kappa}})$ in \eqref{prob:P2} with $\kappa=\ov\kappa$.

Let us denote by $c_{0}>0$ a constant such that $\|m\|_{\infty} \leq c_{0} \|m\|_{1,q}$ for all $m\in W^{1,q}(\Om)$  and by $c_1>0$ the constant in \eqref{estimacion_w1q}. Thus, any solution of $G(m,w)=0$ satisfies that 
$$\|m\|_{\infty}\leq c_0 c_1(\|w\|_{q}+1)\leq 2c_0c_1 \max\{\|w\|_{q},1\}.$$

Now, let us fix a solution $(m_\g, w_\g)$  of \eqref{prob:P2} with $\kappa \equiv \gamma$.  If $\|w_\g\|_{q} \leq 1$, then $\|m_\g\|_{\infty} \leq 2c_0c_1$, so let us assume that $\|w_\g\|_{q} > 1$ and so $\|m_\g\|_{\infty} \leq 2c_0c_1  \|w_\g\|_{q}$. Since $w_{\gamma}$ vanishes a.e. in $\{m_{\g}=0\}$, arguing as in 
\eqref{estimation_with_the_cost} we get that 
\begin{equation}\label{estimation_with_the_cost_1}
\|w_{\g}\|_{q}^q  \leq q C_1^{q-1}\left(\cB_{q}(1/|\Om|,0)+\cF(1/|\Om|)+C_2- C_{\cF} \right)(2c_0c_1  \|w_\g\|_{q})^{q-1}.
\end{equation}
Since \eqref{growthhstar} implies that $\cB_{q}(1/|\Om|,0)\leq C_2$, we find that
$$
\|w_{\g}\|_{q}\leq q C_1^{q-1}\left(\cF(1/|\Om|)+ 2C_2- C_{\cF} \right)(2c_0c_1)^{q-1}
$$
and so 
$$
\|m_{\gamma}\|_{\infty} \leq (2c_0c_1)^{q}q C_1^{q-1}\left(\cF(1/|\Om|)+ 2C_2- C_{\cF} \right).
$$
The result follows. 
\end{proof}
\subsection{Existence of  solutions of the Mean Field Game systems}\label{subsec:q} 
We recall that for a non-empty closed and convex set $K$, the normal cone $\cN_{K}(x)$ to $K$ at $x$ is defined as $\cN_{K}(x):= \partial \chi_{K}(x)$. 
We have the following existence result.
\begin{theorem}\label{existence_MFG_without_density_constraints}   Assume that \eqref{assumption_F} holds and that $\cF$ is bounded from below in $W_{+}^{1,q}(\Om)$. Then, there exists  $(m,u,\lambda) \in W^{1,q}(\Om) \times Y \times \R$ such that 
\begin{equation}\label{MFGq_1_abstract}
\left\{
\begin{array}{rll} 
A^{\ast}u+ H(\cdot,\nabla u) +\lambda &= D\cF(m), & \; \\[6pt]
-\Delta  m - \diver\left(m \nabla_\xi H(\cdot,\nabla u)\right)  &= 0,  & \iin\ \Omega, \\[6pt]
(\nabla m+m \nabla_\xi H(\cdot,\nabla u))\cdot n &= 0,   & \oon\ \partial\Omega,\\[6pt]
\ds\int_\Omega m\dd x=1,&     m(x)> 0, &  \iin\ \ov{\Omega},
\end{array}
\right. 
\end{equation}
where the second equation,  with its boundary condition, is satisfied in the weak sense {\rm(}see \eqref{principalecweakform}{\rm)}. 
\end{theorem}
\begin{proof} Theorem \ref{existence0} yields the existence of a solution  $(m,w) \in W^{1,q}(\Om) \times L^{q}(\Om)^d$ of problem \eqref{prob:P1}.   By definition, 
$$ \hat{\cB}_{q}(m,w)+ \cF(m) \leq  \hat{\cB}_{q}(\tilde m,\tilde w)+ \cF(\tilde m) \hspace{0.3cm} \forall \; (\tilde m,\tilde w)\in W^{1,q}(\Omega) \times L^{q}(\Omega)^{d}, $$
where $\hat{\cB}_{q}(m,w):=  \cB_{q}(m,w)+  \ds \chi_{G^{-1}(0)}(m,w)$. Since $ \hat{\cB}_{q}(m,w)$ is finite, the G\^ateaux differentiability of $\cF$ and the convexity of $\hat{\cB}_{q}$ imply that 
$$ - \langle D\cF(m), \tilde m \rangle \leq \liminf_{ \tau \downarrow 0}  \frac{\hat{\cB}_{q}(m+\tau \tilde m,w+\tau \tilde w)-\hat{\cB}_{q}(m,w)}{\tau}=\inf_{\tau> 0}  \frac{\hat{\cB}_{q}(m+\tau\tilde m,w+\tau\tilde w)-\hat{\cB}_{q}(m,w)}{\tau},$$
where we have denoted by $\langle \cdot, \cdot \rangle$ the duality product between $(W^{1,q}(\Om))^{\ast}$ and $W^{1,q}(\Om)$. Taking  $\tau=1$ in the last term of the previous inequality, we get that 
$$ \hat{\cB}_{q}(m,w)-  \langle D\cF(m), \tilde m  \rangle  \leq \hat{\cB}_{q}(m+\tilde m,w+\tilde w)  \hspace{0.3cm} \forall \; (\tilde m,\tilde w)\in W^{1,q}(\Omega) \times L^{q}(\Omega)^{d},$$
from which $(-D\cF(m),0) \in \partial \hat{\cB}_{q}(m,w)$. Set  $(\hat{m},\hat{w}):=(1/|\Om|, 0)$. Since $\chi_{G^{-1}(0)}(\hat{m},\hat{w})=0$ and $\cB_{q}$ is finite and  continuous at $(\hat{m},\hat{w})$ (since $q>d$), we have (see e.g. \cite[Theorem 9.5.4(b)]{AttButMic}) 
\begin{equation}\label{subdifferentialseparation}
\begin{array}{rcl} (-D\cF(m),0) \in \partial \hat{\cB}_{q}(m,w)&=& \partial  \cB_{q}(m,w) + \partial  \chi_{G^{-1}(0)}(m,w)\\[6pt]
\; &=&\partial  \cB_{q}(m,w) + \cN_{G^{-1}(0)}(m,w).  \end{array}\end{equation}
In particular $\partial \cB_{q}(m,w)\neq \emptyset$ and so, by Theorem \ref{thm:subdiff}{\rm(ii)},   $v = (w/m)\one_{\{m>0\}}\in L^{q}(\Om)^{d}$. Lemma \ref{harnack_neumann}, in the appendix, implies that $m >0$ in $\ov{\Om}$, hence  Theorem \ref{thm:subdiff}{\rm(ii)} implies that 
$$\partial \cB_{q}(m,w)= \left\{\left(-H(\cdot,\nabla H^*(\cdot,-v)), -\nabla H^*(\cdot,-v)\right)\right\}.$$
On the other hand, by \cite[Lemma A.1]{MesSil} we have that $G$ is surjective and so      (see e.g. \cite{BonSha})
$$\begin{array}{rcl}\cN_{G^{-1}(0)}(m,w)&=&\{ DG(m,w)^{*}(\hat{u}, \hat{\lambda}) \; ; \; (\hat{u},\lambda) \in Y \times \R\},\\[6pt]
\; &=&\{(A^{*}\hat{u}+ \hat{\lambda}, B^{\ast}\hat{u})\; ; \; (\hat{u}, \hat{\lambda} ) \in Y \times \R\}.\end{array}$$
%
Therefore,  since $B^{\ast} \hat{u}= -\nabla \hat{u}$, we get the existence of  $(\hat{u}, \hat{\lambda})\in   Y\times \R $ such that
\begin{equation}
  -A^{\ast} \hat{u}+H(\cdot,\nabla H^*(\cdot,-v))-  \hat{\lambda}   =  D\cF(m), \hspace{0.5cm} \nabla H^*(\cdot,-v)  =   -\nabla \hat{u}.
\end{equation}
Thus, defining $u=-\hat{u}$ and  $\lambda=-\hat{\lambda}$  we get the first equation in \eqref{MFGq_1_abstract}. On the other hand, since $w=mv$ and $v= -\nabla_\xi H(\cdot,\nabla u)$, the remaining equations in  \eqref{MFGq_1_abstract} also hold true.  
\end{proof}


The proof of the following result, concerning problem \eqref{prob:P2}, is analogous to the previous one (see also Theorem 4.1 and Corollary 4.2 in \cite{MesSil}), hence we omit it. Notice, however, that the extra assumption on the global lower bound for $\cF$ is not needed, since existence also holds true when we only assume \eqref{assumption_F} (see Theorem \ref{existence0}). 

\begin{theorem}\label{existence_MFG_with_density_constraints}  Assume that \eqref{assumption_F} holds. Then, there exists $(m,u,p,\lambda) \in W^{1,q}(\Om) \times Y \times \sM(\ov\Om)\times \R$ such that 
\begin{equation}\label{MFGq_2_abstract}
\left\{
\begin{array}{rll} 
A^{\ast}u+ H(\cdot,\nabla u)-p+ \lambda &= D\cF(m), &  \;  \\[6pt]
-\Delta  m - \diver\left(m \nabla_\xi H(\cdot,\nabla u)\right)  &= 0,  & \iin\ \Omega, \\[6pt]
(\nabla m+m \nabla_\xi H(\cdot,\nabla u))\cdot n &= 0,   & \oon\ \partial\Omega,\\[6pt]
\ds\int_\Omega m\dd x=1,&   0 < m(x)\le \k(x), \hspace{0.5cm}  &  \iin\ \ov{\Omega},\\[8pt]
 \mbox{{\rm spt}}(p) \subseteq \{m=  \kappa\},  &  p \geq 0, &  \iin\ \ov{\Omega}
\end{array}
\right. 
\end{equation}
where the second equation, with its boundary condition is satisfied in the weak sense {\rm(}see \eqref{principalecweakform}{\rm)}. 
\end{theorem} 
\begin{remark}
In the above theorem, the variable $p$ plays the role of the Lagrange multiplier associated to the density constraint $m\le\kappa$.
\end{remark}
We discuss now the uniqueness of solutions of systems \eqref{MFGq_1_abstract} and \eqref{MFGq_2_abstract} under a convexity assumption on $\cF$.
\begin{proposition}\label{uniqueness} If $\cF$ is strictly convex in $W^{1,q}_{+}(\Omega)$ then the solutions  of \eqref{MFGq_1_abstract}-\eqref{MFGq_2_abstract} are unique.
\end{proposition}
\begin{proof} Let us consider first \eqref{MFGq_1_abstract}. Since $\cF$ is convex, \eqref{prob:P1} is a convex problem and so if $(m,u,\lambda)$ solves \eqref{MFGq_1_abstract} then  $(m,-m \nabla_\xi H(\cdot,\nabla u))$ solves \eqref{prob:P1}. Thus, the uniqueness of $m$ is a straightforward consequence of the strict convexity of $\cF$. Since $m>0$ and for all $x\in  \Om$ the map $\R^d\ni\xi \mapsto  \nabla_{\xi}H(x, \xi)$   is injective and $\R^d\ni w \mapsto H^{\ast}\left(x, \frac{w}{m(x)}\right)m(x)$ is strictly convex we have that $\nabla u$ is unique. Thus, uniqueness of $u$ in $Y$ follows and, as a consequence, the first equation in \eqref{MFGq_1_abstract} yields the uniqueness of $\lambda$. The proof of uniqueness of $(m,u)$ for system \eqref{MFGq_2_abstract} is the same as the previous one.  By considering test functions supported in $\{0<m<\kappa\}$ we get the uniqueness of $\lambda$, from which the uniqueness of $p$ follows. 
\end{proof}
\begin{remark} The previous uniqueness result for solutions of \eqref{MFGq_2_abstract} improves the one stated in \cite[Remark 4.2]{MesSil}.
\end{remark}
Let us detail  the novelty of our results. Compared to \cite{MesSil}, when $q>d$,  we  consider more general Hamiltonians and we prove that $m$ is strictly positive in $\ov{\Om}$ which allows us to eliminate the Lagrange multiplier $\mu \in \sM(\ov\Om)$ from the system $(MFG)_q$ in  \cite{MesSil}. As we have seen, we can also get rid of the density constraint and prove, using variational methods, the existence of solutions of \eqref{MFGq_1_abstract}. Most importantly,  we can consider, for both systems \eqref{MFGq_1_abstract} and \eqref{MFGq_2_abstract}, rather general right-hand sides for the HJB equation,  since we allow $\cF$ to be  non-convex. As an example of a class of functions we can deal with in \eqref{MFGq_1_abstract}, we consider
\begin{equation}\label{eq:F_model}\cF(m):= \int_{\Om} F\left(x,m(x),\nabla m(x)\right) \dd x, 
\end{equation}
where $F: \Omega \times \R \times \R^d \to \R$ is a Carath\'eodory function,  i.e. for all $(z,\xi) \in  \R \times \R^d $ we have $F(\cdot, z,\xi)$ is measurable and for a.e. $x \in \Om$ the function $F(x, \cdot, \cdot)$ is continuous. In addition, suppose that:
\begin{itemize}
\item[(i)] For a.e. $x \in \Om$ and all $z\in \R$ the function $F(x, z, \cdot)$ is convex.
\item[(ii)] For all $R>0$ there exists   $\gamma \in L^{1}(\Om)$   such that for a.e. $x\in \Om$
\begin{equation}\label{conditional_lower_bound}
F(x,z, \xi) \geq    \gamma (x)  \hspace{0.5cm} \forall \; 0\leq z \leq R, \; \; \xi  \in \R^d.
\end{equation}
\item[(iii)]   For a.e. $x\in \Om$ the function $F(x, \cdot, \cdot)$ is differentiable.  Moreover, for all $R>0$ there exists  $a_{0}  \in L^{1}(\Om)$ and $b_0=b_0(R) \geq 0$ such that for a.e. $x\in \Om$
$$
|F(x,z, \xi)| \leq a_{0}(x)+ b_0|\xi|^{q} \hspace{0.5cm} \forall \;  |z| \leq R, \; \; \xi  \in \R^d.
$$ 
\item[(iv)] For all $R>0$ there exists  $a_{1}  \in L^{1}(\Om)$, $a_{2}\in L^{q'}(\Om)$ and  $b_1=b_1(R)\geq 0$  such that for a.e. $x\in \Om$,  $|z|\leq R$  and $ \xi \in \R^{d}$ we have that 
$$\begin{array}{c}
|\partial_{z} F(x,z, \xi)|  \leq a_1(x)+   b_1 |\xi|^{q}, \hspace{0.5cm}|\nabla_{\xi} F(x,z, \xi)| \leq a_2(x)+   b_1  |\xi|^{q-1}.
\end{array}
$$
\end{itemize}
Since $q>d$, assumptions {\rm(i)-\rm(ii)} imply the weak lower semi-continuity of $\cF$ (see e.g. \cite[Corollary 3.24]{Dac}),  conditions  {\rm(ii)}-{\rm(iii)} imply that $\cF(m) \in \R$ for all $m \in  W^{1,q}(\Om) $   and the last condition {\rm(iv)} implies   that $\cF$ is G\^ateaux differentiable at every $m \in  W^{1,q}(\Omega) $ (see e.g. the proof of  \cite[Theorem 3.37]{Dac}). Thus, assumption \eqref{assumption_F} for $\cF$ is satisfied in this case. The  G\^ateaux derivative of $\cF$ at $m\in  W^{1,q}(\Om) $ is given by  linear continuous functional 
$$
\langle D\cF(m),z\rangle= \int_{\Om} \left[\partial_{z} F\left(x,m(x),\nabla m(x)\right)z(x)+\nabla_{\xi} F\left(x,m(x),\nabla m(x)\right) \cdot \nabla z(x) \right] \dd x,
$$
for all $z\in W^{1,q}(\Om)$. We obtain the following corollary of Theorem \ref{existence_MFG_without_density_constraints} and of Theorem \ref{existence_MFG_with_density_constraints} when  $F$ is independent of $\nabla m$. 
\begin{corollary}\label{existence_standard_MFGs} Let $f: \Om \times \R \to \R$  be a Carath\'eodory function. Assume that for all $R>0$ there exist $a$, $\gamma_1 \in L^{1}(\Omega)$ such that for  a.e. $x\in \Omega$
\begin{equation}\label{conditions_on_f}
\begin{array}{l}
\int_{0}^{z}f(x,z')\dd z' \geq \gamma_1(x)  \hspace{0.3cm} \mbox{and } \; \; |f(x, z)| \leq  a(x) \hspace{0.3cm} \forall \;  |z| \leq R. \\[6pt]
  \hspace{0.3cm}  
\end{array}\end{equation} 

Then,   \eqref{eq:MFG_2} admits at least one solution $(m,u,p,\lambda) \in W^{1,q}(\Om) \times Y \times \sM(\ov{\Om}) \times \R$. If, in addition, there exists   $\gamma_2 \in L^{1}(\Omega)$ such that $\int_{0}^{z}f(x,z')\dd z' \geq \gamma_2 (x)$ for a.a. $x\in \Omega$ and $z \geq 0$, 
then system \eqref{eq:MFG_1} admits at least one solution  $(m,u,\lambda) \in W^{1,q}(\Om) \times Y \times \R$. In both cases, we have  the additional regularity  $u\in W^{1,s}(\Om)$   for all $s\in(1,d/(d-1)).$ 
\end{corollary}
\begin{proof} Consider $F(x,z):=\int_{0}^{z}f(x,z')\dd z'$ in \eqref{eq:F_model}. Since \eqref{conditions_on_f} implies \eqref{assumption_F}, existence of a solution $(m,u,p,\lambda) \in W^{1,q}(\Om) \times Y \times \sM(\ov{\Om}) \times \R$ of \eqref{eq:MFG_2} follows directly from Theorem \ref{existence_MFG_with_density_constraints}. Analogously,  $\int_{0}^{z}f(x,z')\dd z' \geq \gamma_2 (x)$ for all $m\geq 0$ implies that $\cF$ has a global lower bound on $W^{1,q}_{+}(\Om)$.  Hence, existence of a solution  $(m,u,\lambda) \in W^{1,q}(\Om) \times Y \times \R$ of \eqref{eq:MFG_1} follows directly from Theorem \ref{existence_MFG_without_density_constraints}. Since in both systems $m \in C(\ov{\Om})$, assumption \eqref{conditions_on_f} also implies that $f(x,m(x)) \in L^{1}(\Omega)$ and the $ W^{1,s}$ regularity for $u$ ($s \in(1,d/(d-1))$), in both systems, follows from  \cite[Th\'eor\`eme 9.1]{stampacchia3}.
\end{proof}
Now, we comment on some other possible choices of $\cF$. 

\begin{remark}{\rm(i)} A simple and interesting example is given by $f(x,z)= |z|^{r}$, where $r>0$. In this case, $F(x, z)= \frac{1}{r+1}|z|^{r+1}$ if $z>0$ and $F(x, z)= -\frac{1}{r+1}|z|^{r+1}$, otherwise. Thus, the assumptions of Theorem \ref{existence_standard_MFGs} and Proposition \ref{uniqueness} are satisfied and so systems \eqref{eq:MFG_1} and \eqref{eq:MFG_2} admits   unique solutions. Note that the growth of $f$ is arbitrary, showing one the advantages  of the variational approach {\rm(}compare with  \cite{MR3160525,MR3333058,pimentel2015regularity} where the growth of $f$ is restricted{\rm)}.  On the other hand, if we consider $f(x,z)= -|z|^{r}$, its primitive is not bounded from below in $[0,\infty)$ and   the assumptions of Theorem \ref{existence_standard_MFGs} are not satisfied for problem \eqref{eq:MFG_1}. However, they are satisfied for  \eqref{eq:MFG_2} and the existence of at least one solution to  \eqref{eq:MFG_2} is ensured also in this case. \smallskip\\
{\rm(ii)} In order to exemplify the possible dependence of $\cF$ on $\nabla m$,  let us  take $\cF(m):= \frac12 \int_{\Om}|\nabla m|^{2} \dd x$. In this case Theorem \ref{existence_MFG_without_density_constraints} yields the existence of weak  solution of 
$$
\left\{
\begin{array}{rll} 
- \Delta u+ H(\cdot,\nabla u) +\lambda &= -\Delta m ,  &  \iin\ \Om, \\[6pt]
\nabla (u-m)\cdot n &= 0 ,   & \oon\ \partial \Om, \\[6pt]
-\Delta m - \diver\left(m \nabla_\xi H(\cdot,\nabla u)\right)  &= 0,  & \iin\ \Omega, \\[6pt]
(\nabla m+m \nabla_\xi H(\cdot,\nabla u))\cdot n &= 0,   & \oon\ \partial\Omega,\\[6pt]
\ds\int_\Omega m\dd x=1,&     m(x)> 0, &  \iin\ \ov{\Omega}.
\end{array}
\right.
$$
Moreover, by Proposition \ref{uniqueness}, the solution is unique. \smallskip\\
{\rm(iii)}  We can also consider a non-local dependence on $m$. For instance, setting $\Omega_{\varepsilon}:=\{x \in \Om \; ; \; d(x,\partial \Om)>\varepsilon\}$, where $d(\cdot, \partial \Om)$ is the distance function to $\partial \Om$,  we can take $F$ as before and
$$ \cF(m):= \int_{\Omega_{\varepsilon}} F(x, \rho_0 \ast m, \rho_{1}  \ast \partial_{x_1} m(x), ..., \rho_{d}  \ast \partial_{x_d} m(x)) \dd x, $$
for some given regular kernels $\rho_0,..., \rho_d$ supported on  $B_\e(0)$. In this case, it is easy to check  that \eqref{assumption_F} holds  without requiring the convexity of $F(x,z, \cdot)$. 
\end{remark}

Now let $f: \Om \times \R \to \R$  be a Carath\'eodory function satisfying that there exists $\gamma_2 \in L^{1}(\Om)$ such that $\int_{0}^{z}f(x,z')\dd z' \geq \gamma_2 (x)$ for  a.e.  $x\in \Omega$ and $z \geq 0$. Moreover, assume that for all $R>0$ there exists $a \in L^{1}(\Omega)$ such that for  a.e.  $x\in \Omega$ and $ |z| \leq R$ we have that $|f(x, z)| \leq  a(x)$. Under these assumptions, Corollary \ref{existence_standard_MFGs}  ensures the existence of at least one solution $(m,u,\lambda) \in W^{1,q}(\Om) \times Y \times \R$ of \eqref{eq:MFG_1}. Using a bootstrapping argument, we show in the next result some  additional  local regularity  properties  for $(m,u)$.

\begin{proposition}\label{regularity_boots} 
Consider the above setting and suppose that  $\Om\ni x\mapsto f(x,m(x)) \in \R$ belongs to $L^r(\Om)$ for some $r > d$ and that $\nabla_{\xi}H(x, \cdot)$ is H\"older continuous, uniformly on $x\in \Om$. Then, there exist    $\alpha_0$, $\alpha_1 \in (0,1)$  such that 
\begin{equation}\label{reg1} 
u\in C^{1,\a_0}_{\rm{loc}}(\Omega)\;\;\; {\rm{and}}\;\;\; m\in C^{1,\a_1}_{\rm{loc}}(\Omega).
\end{equation}
\end{proposition} 

\begin{proof}
{\it Step 1.} We show that there exists $k>d$ such that $u\in W_{\rm{loc}}^{2,k}(\Omega)$.  By the classical Sobolev embeddings, this implies  that $u\in C_{\rm{loc}}^{1,\a_0}(\Omega)$ (for some $ \a_0  \in (0,1)$).   Let  $r_1\in (q',d/(d-1))$. By Corollary  \ref{existence_standard_MFGs},   we have  that $|\nabla u|^{q'}\in L^{r_1/q'}(\Om)$ and so, by \eqref{hyp:1}, we have that $H(\cdot,\nabla u)\in L^{r_1/q'}(\Om)$. Furthermore, since $f(\cdot, m(\cdot))\in L^r(\Om)$ and $r>\delta_1:=r_1/q'$,   the classical regularity theory for elliptic equations (see \cite{gilbarg}) implies that $u\in W_{\rm{loc}}^{2,\d_1}(\Om)$. In particular, the Sobolev inequality (see e.g. \cite{adams}) yields $u\in W_{\rm{loc}}^{1,\frac{d\d_1}{d-\d_1}}(\Omega)$ and so $|\nabla u|^{q'}\in L_{\rm{loc}}^{\delta_2}(\Omega)$ with $\delta_2:=\frac{d\d_1}{q'(d-\d_1)}$. We easily check that $\delta_2>\delta_1$ and so we improve the regularity of $u$ to obtain that $u\in W^{2,\min\{r,\delta_2\}}_{\rm{loc}}(\Omega)$. Thus, if  $\delta_2 > d$ we obtain the first relation in \eqref{reg1}. Otherwise, for $i\geq 2$, inductively we define  the sequence $\d_{i+1}:=  \frac{d\d_i}{(d-\d_i)q'}$. Since $\delta_{i+1}-\delta_{i} \geq (q'+d-dq')/(d-\delta_i)q'$ and $q'+d-dq'>0$, we get that 
$\delta_{i+1}-\delta_{i} \geq (q'+d-dq')/ d q'$ if $d> \delta_i$. Therefore,  after a finite number of steps we get the existence of $  i^*\ge 2$ such that $\d_{i^*}>d$ and $u\in W_{\rm{loc}}^{2,k}(\Omega)$ with $k:=\min\{r,\d_{i^*}\}>d$.   \smallskip\\
{\it Step 2.} Let us prove that $m\in C_{\rm{loc}}^{1,\a_1}(\Omega)$ for some $\a_1 \in (0,1)$. Since $m\in W^{1,q}(\Om)$ and $q>d$, we already have that $m$ is  H\"older continuous. Having $u\in C_{\rm{loc}}^{1,\a_0}(\Om)$, this implies that  $\nabla u\in C_{\rm{loc}}^{0,\a_0}(\Om)^d,$ hence $m\nabla_\xi H(\cdot,\nabla u)\in C_{\rm{loc}}^{0,\hat{\a}}(\Om)^d$, for some $\hat{\a}\in (0,1)$. Using a Schauder-type estimate (see  \cite[Theorem 5.19]{GiaMar}) we get that $m\in C_{\rm{loc}}^{1,\a_1}(\Om)$ for some $\a_1 \in (0,1)$. 
%
%
\end{proof}
\begin{remark}\label{generalization_regularity_density}  If $f$ satisfies    \eqref{conditions_on_f}, with $a\in L^{r}(\Om)$ for some $r>d$,    and   $(m,u,p,\lambda)$ solves \eqref{eq:MFG_2}, we have that $(u,m)$ admits the   regularity \eqref{reg1}  locally in the open set $\{0 < m < \kappa\}$ {\rm(}see \cite[Proposition 4.3]{MesSil} for a similar result in a simpler case{\rm)}. 
\end{remark}
%
%

\section{An application to multipopulation systems}\label{multipopulation_section} In this section we  show a simple application of our results to the study of systems of the form
\begin{equation}\label{MFGq_multipopulation}
\left\{
\begin{array}{rll} 
- \Delta u_{i}+ H^i(\cdot,\nabla u_{i})+ \lambda_i &= f^i(x,(m_{i})_{i=1}^{N}) ,  &  \iin\ \Om, \\[6pt]
\nabla u_{i}\cdot n &= 0 ,   & \oon\ \partial \Om, \\[6pt]
-\Delta m_{i} - \diver\left(m_{i} \nabla_\xi H^i(\cdot,\nabla u_{i})\right)  &= 0,  & \iin\ \Omega, \\[6pt]
(\nabla m_{i}+m_{i} \nabla_\xi H^i(\cdot,\nabla u_{i}))\cdot n &= 0,   & \oon\ \partial\Omega,\\[6pt]
\ds\int_\Omega m_{i}\dd x=1,&     m_{i}(x)> 0, &  \iin\ \ov{\Omega},\\[6pt]
i=1, ..., N,
\end{array}\tag{MFG$_N$}
\right.  
\end{equation}
where $N\geq 2$ (see \cite{MR3333058,MR3530210}). Here $(m_{i})_{i=1}^{N}$ describes the densities of $N$ populations. The Hamiltonians $H^i:\Om\times\R^d\to\R$ ($i=1,\dots,N$) are supposed to satisfy the assumptions in Section \ref{prelim} (see in particular \eqref{hyp:1}). The given functions $f^i:\Om\times[0,+\infty)^N\to\R$ ($i=1,\dots,N$) are such that  for all $\zeta \in \R^{N}$ the function $f^i(\cdot, \zeta)$ is measurable and for a.e. $x\in \Omega$ the function  $f^{i}(x, \cdot)$ is continuous.  Suppose that  

\begin{equation}\label{multipopulation_hyp1}\begin{array}{c}
\exists \; \gamma_i \in L^{1}(\Om) \; \; \mbox{such that } \int_{0}^{z}f^{i}(x, z_i, (\zeta_j)_{j\neq i})\dd z_i \geq \gamma_i (x) \\[6pt]
 \mbox{for  a.e.  } x\in \Omega, \; \; \forall \;  z \geq 0, \; \; \forall \; (\zeta_1, ..., \zeta_{i-i}, \zeta_{i+1},..., \zeta_N) \in [0,+\infty)^{N-1}, 
\end{array}
\end{equation}
where we have denoted
$$ f^{i}(x, z_i, (\zeta_j)_{j\neq i}):= f^{i}(x,\zeta_1, ..., \zeta_{i-i}, z_i, \zeta_{i+1},..., \zeta_N).
$$
 Moreover, we assume that   for all $i=1,\dots,N$
\begin{equation}\label{multipopulation_hyp2}\begin{array}{c}
\forall \; R>0, \; \; \exists \; a_{i} \in L^{1}(\Omega) \; \; \mbox{such that } \; \; |f^{i}(x, z, (\zeta_j)_{j\neq i})|  \leq a_{i}(x), \\[6pt]
 \mbox{for  a.e. } x\in \Omega, \; \; \forall \;  | z|  \leq R, \; \; \forall \; (\zeta_1, ..., \zeta_{i-i}, \zeta_{i+1},..., \zeta_N) \in [0,+\infty)^{N-1},  
\end{array}
\end{equation}
and that 
\begin{equation}\label{multipopulation_hyp3}\begin{array}{c}
\forall \;   (\zeta_1, ..., \zeta_{i-i},  \zeta_{i+1},..., \zeta_N) \in [0,+\infty )^{N-1} \\[6pt]
\mbox{the map  $z \in [0,+\infty) \to f^{i}(x, z, (\zeta_j)_{j\neq i}) \in \R$ is non-decreasing.} 
\end{array}
\end{equation}
We have the following result:
\begin{proposition}\label{existence_result_multipopulation_I} Suppose that for all $i=1,\dots,N$ the function $f^i$ satisfies \eqref{multipopulation_hyp1}, \eqref{multipopulation_hyp2} and \eqref{multipopulation_hyp3}. Then, system  \eqref{MFGq_multipopulation} admits at least one solution $m=(m_1,..., m_{N})$, $u=(u_{1},...,u_N)$ and $\lambda=(\lambda_1,...,\lambda_N)$, where, for all $i=1,..., N$,  $m_{i} \in W^{1,q}(\Om)$ and $u_i \in W^{1,s}(\Om)$ {\rm(}for all $s \in (1, d/(d-1)${\rm)}.
\end{proposition}
\begin{proof} Let us define $\cF^i: W^{1,q}(\Om)^{N} \to \R$ as
$$\begin{array}{c}
\cF^i(m_i, (m_{j})_{j\neq i}):= \int_{\Om} F^i(x, m_i(x), (m_j(x))_{j\neq i})\dd x\\[6pt] 
\mbox{where } 
F^i(x, z, (\zeta_j)_{j\neq i}):= \int_{0}^{z}f^{i}(x, z_i, (\zeta_j)_{j\neq i})\dd z_i. 
\end{array}
$$
Consider the space $X:= \left(W^{1,q}(\Om)\times L^{q}(\Om)^{d}\right)^{N}$, endowed with the weak-topology,  and  the set-valued map $T: X\to 2^{X}$ defined as
$$
 T(x):= \prod_{i=1}^N\mbox{argmin}_{(m',w') \in W^{1,q}(\Om) \times L^{q}(\Om)^d}\left\{ \cB_{q}(m',w') + \ds \chi_{G^{-1}(0)}(m',w')+\cF^{i}(m', (m_j)_{j\neq i})\right\},
$$
where $x= \left( (m_1,w_1), ..., (m_{N}, w_{N}) \right)$.  By Theorem \ref{existence0}, the embedding $W^{1,q}(\Om)\hookrightarrow C(\ov{\Om})$ and our assumptions, we have that $T(x)$ is a non-empty compact set for all $x\in X$. Assumption \eqref{multipopulation_hyp3} implies that $T(x)$ is also convex. Moreover, by \eqref{estimates_solution_optimization_problem} and \eqref{multipopulation_hyp2}, for $x\in X$  we have the existence of $c>0$, independent of $x$, such that $\|\bar{m}_i\|_{1,q} \leq c$ and $\|\bar{w}_i\|_q \leq c$ for all $\bar{m}=(\bar{m}_1,...,\bar{m}_{N})$ and  $\bar{w}=(\bar{w}_1,...,\bar{w}_N)$ such that $(\bar{m},\bar{w}) \in T(x)$. Therefore, defining 
$$
K_{c}:= \{(m,w) \in X \; ; \; \|m_i\|_{1,q} \leq c, \; \; \|w_i\|_q \leq c, \; \; \forall \; i=1,...,N\},
$$  
we have that $T(K_c) \subseteq K_c$. Now, let us prove that $T$ is upper-semicontinuous, i.e. $T^{-1}(M):=\{x'\in X \; ; \; T(x')\cap M \ne \emptyset\}$ is closed for all closed sets $M\subseteq X$. Indeed, let  $x^n=( (m_1^n, w_1^n),...., (m_N^n, w_N^n)) \in T^{-1}(M)$ such that $x_n \to x=( (m_1, w_1),...., (m_N, w_N))$. By definition, we have the existence of  $\bar{x}^n=( (\bar{m}_1^n, \bar{w}_1^n),...., (\bar{m}_N^n, \bar{w}_N^n)) \in M$ such that
$$\begin{array}{c}
\cB_{q}(\bar{m}_i^n,\bar{w}_i^n) + \ds \chi_{G^{-1}(0)}(\bar{m}_i^n,\bar{w}_i^n)+\cF^{i}(\bar{m}_i^n, (m_j^n)_{j\neq i}) \\[6pt]
\leq \cB_{q}(m',w') + \ds \chi_{G^{-1}(0)}(m',w')+\cF^{i}(m', (m_j^n)_{j\neq i}) \; \; \forall \; (m',w') \in W^{1,q}(\Om)\times L^{q}(\Omega)^d.
\end{array}
$$
By  \eqref{estimates_solution_optimization_problem} we have that $\bar{x}_n$ is bounded in  $\left(W^{1,q}(\Om)\times L^{q}(\Om)^{d}\right)^{N}$  and so, up to some subsequence,  there exists $\bar{x}=( (\bar{m}_1, \bar{w}_1),...., (\bar{m}_N, \bar{w}_N))$ such that $\bar{x}^n\to \bar{x}$ in $X$ and so, since $M$ is closed, $\bar{x}\in M$. Under our assumptions,  the Lebesgue's dominated convergence theorem implies the   weak continuity of  $ \cF^{i}$ in $W^{1,q}(\Om)^{N}$,  and so we can pass to the limit to obtain that $x\in T^{-1}(M)$. By Kakutani fixed-point theorem, there exists $x=((m_1,w_1),..., (m_{N},w_{N}))$ such that $x \in T(x)$. The result follows from Corollary \ref{existence_standard_MFGs}.
\end{proof}
\begin{remark} 
{\rm(i)} As we pointed out, the result in Proposition \ref{existence_result_multipopulation_I} is a simple consequence of the variational method we presented in the previous sections.  We refer the reader to \cite{MR3333058,cirant_verzini,MR3530210,MR3597009} for a more detailed study, and sharper results, based on PDE arguments tackling directly system \eqref{MFGq_multipopulation}.
\smallskip\\
{\rm(ii)} The local regularity results presented in Proposition \ref{regularity_boots} for the one-population case directly extend to the solutions of system \eqref{MFGq_multipopulation}.
\end{remark}

We can also consider the instance of \eqref{MFGq_multipopulation} where the functions $f^{i}$ ($i=1,...,N$) satisfy that there exists a  Carath\'eodory function function $F: \Omega \times \R^{N} \to \R$ such that  for a.e. $x\in \Om$ the function $F(x,\cdot)$ is differentiable  and for all $i=1,...,N$ we have that  
\begin{equation}\label{global_potential}
f^{i}(x,\zeta_i,(\zeta_j)_{j\neq i})= \partial_{\zeta_i} F(x,\zeta) \; \; \; \mbox{for  a.e.  $x\in \Omega$, $\forall \; \zeta=(\zeta_1, ..., \zeta_N)\in \R^{N}$.} 
\end{equation}
As suggested in \cite[Remark 15]{MR3333058}, in this case system \eqref{MFGq_multipopulation} can be found as the optimality condition of the optimization problem 
\begin{equation}\label{prob:PN}
\inf_{m_i \in W^{1,q}(\Om), \; w_i\in L^{q}(\Om)^d \atop
i=1,...,N} \; \sum_{i=1}^{N} \left[\cB_{q}(m_{i},w_{i}) + \chi_{G^{-1}(0)}(m_{i},w_{i})\right] + \int_{\Om}F(x,m_1(x),...,m_{N}(x))\dd x. \tag{P$_N$}
\end{equation}
Indeed, suppose that  $F$ satisfies that there exists  $\gamma \in L^{1}(\Om)$ such that 
\begin{equation}\label{condition_F_symetric_caseI}
F(x, \zeta) \geq \gamma(x) \hspace{0.3cm} \mbox{ for a.a. $x\in \Om$ and for all $\zeta \in \R^N$}. \end{equation}
Moreover, suppose that  for all $R>0$ there exists $a \in L^{1}(\Om)$ such that 
\begin{equation}\label{condition_F_symetric_caseII}
|F(x,\zeta)| \leq a(x) \hspace{0.3cm} \mbox{for a.a. $x\in \Om$ and $\zeta \in \R^{N}$ such that $|\zeta_i| \leq R$ for all $i=1,...,N$}.
\end{equation}
Then,  arguing as in the proof of Theorem \ref{existence0}, we get the existence of a solution $m=(m_1,...,m_N)$, $w=(w_1,...,w_N)$ of \eqref{prob:PN}, and so, mimicking the proof of Theorem \ref{existence_MFG_without_density_constraints}, we get the following result:
\begin{proposition}\label{multipopulation_result_I} Suppose that $f^{i}$, $i=1,...,N$, satisfy \eqref{global_potential}, with $F$ satisfying \eqref{condition_F_symetric_caseI}-\eqref{condition_F_symetric_caseI}. Then, system  \eqref{MFGq_multipopulation} admits at least one solution $m=(m_1,..., m_{N})$, $u=(u_{1},...,u_N)$ and $\lambda=(\lambda_1,...,\lambda_N)$, where, for all $i=1,..., N$,  $m_{i} \in W^{1,q}(\Om)$ and $u_i \in W^{1,s}(\Om)$ {\rm(}for all $s \in (1, d/(d-1)${\rm)}.
\end{proposition}

Note that \eqref{global_potential} is restrictive. On the other hand, the previous result does not require the strong boundedness condition \eqref{multipopulation_hyp2} and the monotonicity assumption 
\eqref{multipopulation_hyp3}. Moreover, this framework allows us to introduce density constraints of the form $m \in \cK$, where 
$$
\cK:=  \left\{ m \in W^{1,q}(\Om)^{N} \; ; \; \sum_{i=1}^{N}\alpha_{i}m_{i}(x) \leq \kappa(x) \right\} .
$$
We suppose that  $\kappa \in W^{1,q}(\Om)$ satisfies  $\kappa(x) >0$, for all $x\in \ov{\Om}$,  and the {\it weights} $(\alpha_i)_{i=1}^N$ satisfy 
\begin{equation}\label{conditions_on_the_weights}
\alpha_i \geq 0, \; \; \forall \; i=1,...,N, \; \; \exists \; \bar{i}\in \{1,..., N\} \; \mbox{such that $\alpha_{\bar{i}}>0$} \; \;  \mbox{and } \sum_{i=1}^{N} \alpha_i < \|\kappa\|_{1}.  
\end{equation} 
Condition \eqref{conditions_on_the_weights} implies that if for all $i=1,...,N$ we define $\hat{m}_i:= \kappa/\|\kappa\|_1$ and $\hat{w}_i:= \hat{w}$, where $\hat{w}$ is such that $B\hat{w} =-A\kappa/\|\kappa\|_1$ (we know that such $\hat{w}$ exists by \cite[Lemma A.1]{MesSil}), then $\hat{m}:=(\hat{m}_1,..., \hat{m}_{N})$, $\hat{w}:=(\hat{w}_1, ..., \hat{w}_{N})$ are feasible for problem \small
\begin{equation}\label{prob:PN_alternative}\begin{array}{l}
\inf_{m_i \in W^{1,q}(\Om), \; w_i\in L^{q}(\Om)^d \atop
i=1,...,N} \; \sum_{i=1}^{N} \left[\cB_{q}(m_{i},w_{i}) + \chi_{G^{-1}(0)}(m_{i},w_{i})\right] + \int_{\Om}F(x,m_1(x),...,m_{N}(x))\dd x, \tag{$P'_N$}\\[6pt]
\mbox{s.t. } \hspace{3cm} m \in \cK.
\end{array}
\end{equation} \normalsize 
and $\hat{m}$ is an interior point to the constraint $m \in \cK$, i.e. $\sum_{i=1}^{N}\alpha_i \hat{m}_i (x) <\kappa(x)$ for all $x\in \ov{\Om}$. Therefore, we can argue as in the proof of Theorem \ref{existence0} to show the existence of at least one solution of \eqref{prob:PN_alternative} and then, following the proof of Theorem \ref{existence0} (see also the proofs of Theorem 4.1 and Corollary 4.2 in \cite{MesSil}), we get the following result:
\begin{proposition}
Suppose that for all $i=1, \hdots, N$ the function $f^i$ satisfies the assumptions of Proposition \ref{multipopulation_result_I}. Moreover, assume that \eqref{conditions_on_the_weights} holds. Then system  
\begin{equation}\label{MFGq_multipopulation_case_density_constraints}
\left\{
\begin{array}{rll} 
- \Delta u_{i}+ H^i(\cdot,\nabla u_{i}) -\alpha_i p+\lambda_i &= f^i(x,(m_{i})_{i=1}^{N}), &  \iin\ \Om, \\[6pt]
\nabla u_{i}\cdot n &= 0,  & \oon\ \partial \Om, \\[6pt]
-\Delta m_{i} - \diver\left(m_{i} \nabla_\xi H^i(\cdot,\nabla u_{i})\right)  &= 0,  & \iin\ \Omega, \\[6pt]
(\nabla m_{i}+m_{i} \nabla_\xi H^i(\cdot,\nabla u_{i}))\cdot n &= 0,   & \oon\ \partial\Omega,\\[6pt]
\ds\int_\Omega m_{i}\dd x=1,   &     m_{i}(x)> 0,  &  \iin\ \ov{\Omega},\\[6pt]
i=1, ..., N,
\end{array}\tag{$MFG'_N$}
\right.  
\end{equation}
with 
$$
\sum_{i=1}^{N}\alpha_im_i(x)\leq  \kappa(x) \; \; \mbox{{\rm for all} $x\in \ov{\Om}$}, \; \; \;  p \geq 0 \; \; \mbox{{\rm and }}  \mbox{{\rm spt}}(p) \subseteq  \left\{x \in \ov{\Om} \; ; \; \sum_{i=1}^{N}\alpha_im_i(x)=  \kappa(x)\right\} ,
$$
admits at least one solution $m=(m_1,..., m_{N})$, $u=(u_{1},...,u_N)$, $\lambda=(\lambda_1,...,\lambda_N)$ and $p$, where, for all $i=1,..., N$,  $m_{i} \in W^{1,q}(\Om)$, $u_i \in W^{1,s}(\Om)$ {\rm(}for all $s \in (1, d/(d-1)${\rm)} and $p \in  \sM(\ov{\Om}) $.
\end{proposition}

\vspace{0.3cm}

{\sc Acknowledgements}

\vspace{0.3cm}

The second author was partially  supported by the ANR (Agence Nationale de la Recherche)  project ANR-16-CE40-0015-01 and by the Gaspar Monge Program for Optimization and Operation Research (PGMO) via the project {\it PASTOR}. Both authors were partially supported by the PGMO via the project {\it VarMFGPDE}. The authors also thank M. Cirant for suggesting them to look carefully whether some Harnack inequality holds in their framework. This observation led the authors to find the useful references in \cite{MR3443169} and improve the results in the present paper.

\bibliographystyle{alpha}
\bibliography{alparsilva}{}

\appendix
\section{ }
We prove in this appendix a lemma which is crucial to show the strict positivity of the densities in \eqref{eq:MFG_1} and \eqref{eq:MFG_2}. 

\begin{lemma}\label{harnack_neumann}  Assume  that  $\Om\subset\R^d$ satisfies the assumptions at the beginning of Section \ref{prelim},  $q>d$ and let $v \in L^{q}(\Om)^{d}$. Suppose that  $m \in W^{1,q}(\Om)$ is a weak solution of 
$$\left\{
\begin{array}{rcl}
-\Delta m + \mbox{{\rm div}}(vm) &=& 0 \hspace{0.5cm} \mbox{ {\rm in} } \Omega, \\[4pt]
 \left(\nabla m -bm\right)\cdot n &=& 0 \hspace{0.5cm} \mbox{ {\rm on} } \Omega. \end{array}\right.
$$
Then, $m(x) >0$ for all $x\in \ov{\Om}$
\end{lemma}
\begin{proof}Since $v \in L^{q}(\Om)^{d}$ ($q>d$), by the Harnack's inequality proved in \cite{MR0369884} (see also \cite[Corollary 1.7.2]{MR3443169}) we have that $m(x)>0$ for all $x\in \Om$. It remains to study the positivity of $m$ on the boundary $\partial \Om$. To achieve this, we use a standard reflection argument following \cite{Nit}. Let $x\in \partial \Om$.  Since $\Om$ is supposed to be Lipschitz domain, in a neighborhood of $x$  the boundary $\partial\Om$ can be represented as the graph of a Lipschitz function $\psi:\R^{d-1}\to\R$. Without loss of generality for these local considerations we may suppose that $x=0$ and that
$$B^d_r(x)\cap\partial\Omega=\{(y,\psi(y)): y\in B^{d-1}_r(0)\},$$ 
for some $r>0$ small (otherwise one can perform some isometric transformations on $\Om$ under which all the hypotheses that we assumed on the data remain invariant). We have denoted by $B^{d}_r(z)$ and $B^{d-1}_r(z)$  the balls of radius $r$ centered at $z$ in $\R^d$ and $\R^{d-1}$, respectively. 

Consider a strip-like domain
$$G:=\left\{ (y,\psi(y)+s): y\in B^{d-1}_r(0), s\in(-r,r) \right\},$$
 such that 
$$\Om\cap G=\left\{ (y,\psi(y)+s): y\in B^{d-1}_r(0), s\in(0,r) \right\}.$$
Defining $T:B^{d-1}_r(0)\times(-r,r)\to G$ as $T(y,s):=(y,\psi(y)+s),$ we have that $T$ is a bi-Lipschitz map (an injective Lipschitz continuous map whose inverse is also Lipschitz continuous) and 
$$DT(y, s)=\left[
\begin{array}{cc}
I_{d-1} & 0\\
\nabla\psi(y)^\top & 1
\end{array}
\right]\ \ \ \text{and}\ \ \ DT(y, s)^{-1}=\left[
\begin{array}{cc}
I_{d-1} & 0\\
-\nabla\psi(y)^\top & 1
\end{array}
\right],$$
for a.e. $(y,s)\in B_r^{d-1}(0)\times(-r,r)$ ($I_{d-1}\in \R^{(d-1)\times(d-1)}$ denotes the identity matrix). Define the {\it reflection map} $S:G\to G$ as 
\begin{equation}\label{reflection_definition}
S(T(y,s))=T(y,-s),
\end{equation}
 which clearly leaves the points on $\partial\Om\cap G$ invariant and satisfies that $S(S(x))=x$ for all $x\in G$. Differentiating  both sides of \eqref{reflection_definition} yields  for a.e. $(y,s)\in B_r^{d-1}(0)\times(-r,r)$ 
$$
DS(T(y,s))DT(y,s)=DT(y,-s)\left[
\begin{array}{cc}
I_{d-1} & 0\\
0_{\R^{d-1}}^\top & -1
\end{array}
\right],
$$
from where we get
$$
DS(T(y,s))=DT(y,-s)\left[
\begin{array}{cc}
I_{d-1} & 0\\
0_{\R^{d-1}}^\top & -1
\end{array}
\right]
DT(y,s)^{-1}=\left[
\begin{array}{cc}
I_{d-1} & 0\\
2\nabla\psi(y)^\top & -1
\end{array}
\right].
$$
Thus, $\det (DS(x))=-1,$ $DS(x)^{-1}=DS(x)$, for every $x\in G$,  and $DS(\cdot)$ is bounded on $G$. Note that since $DS(T(y,s))$ does not depend on $s$ we have that $DS(Sx)=DS(x).$

Let us introduce some  notations. We set $U:=G\cap\Omega,$ $V:=G\setminus\ov\Om=S(U)$ and for $b:U\to\R$, we define $\ov b:V\to U$ as $\ov b(x):=b(S(x))$ and
$$
\tilde b(x)=\left\{
\begin{array}{ll}
b(x), & x\in U,\\
\ov b(x), & x\in V,
\end{array}
\right.
$$    
the a.e. extension of $b$ to $G$ (the definition on $\partial \Omega \cap G$ is irrelevant since this set is $\sL^d$-negligible). Arguing as in \cite[Lemma 3.3]{Nit}, if $b\in W^{1,q}(U)$, we have that $\tilde b\in W^{1,q}(G)$, $\nabla \tilde b=\nabla b\one_U+\nabla\ov b\one_V$ and $\nabla \ov b(x)=DS(x)\nabla b (S(x))$ for a.e. $x\in G$.

Now, since   $(m,v)$ satisfies  \eqref{principalecweakform} (with $w=vm$) for tests functions $\varphi \in C^{1}(\ov \Omega)$, defining $\hat v \in L^{q}(G)^{d}$ as
$$
\hat v(x)=\left\{
\begin{array}{ll}
v(x), & x\in U,\\
DS(x) \ov v(x), & x\in V,
\end{array}
\right.
$$ 
(where $\ov v$ is understood componentwise), the pair $(\tilde m,\hat v)$ solves a similar equation on $G$ with test functions $\varphi\in C_c^1(G)$. Indeed, let us take  $\varphi\in C_c^1(G)$ and compute
\begin{equation}
\begin{array}{ll}
\int_G\left( \nabla\tilde m\cdot \nabla\varphi-\tilde m \hat v\cdot\nabla\varphi\right) \dd x & =\int_U \left( \nabla\tilde m\cdot \nabla\varphi-\tilde m \hat v\cdot\nabla\varphi\right) \dd x + \int_V \left( \nabla\tilde m\cdot \nabla\varphi-\tilde m \hat v\cdot\nabla\varphi\right) \dd x\\[8pt]
\; &= \int_U \left( \nabla m\cdot \nabla\varphi- mv\cdot\nabla\varphi\right) \dd x\\[8pt]
\; & \; \; \; + \int_V \left( \nabla \ov m(x)\cdot \nabla\varphi(x)-\ov m(x)\nabla\varphi(x)^\top DS(x)\ov v(x)\right) \dd x\\[8pt]
\;&= \int_V  \nabla\varphi(x)^\top DS(x)\nabla m(S(x)) \dd x\\[8pt]
\; &\; \; \; - \int_V m(S(x)) \nabla\varphi(x)^\top DS(x) v(S(x))  \dd x\\[8pt]
\;&=\int_U \nabla\varphi(S(y))^\top DS(S(y))\nabla m(y) \dd y \\[8pt]
\; &\; \; \; - \int_{U} m(y) \nabla\varphi(S(y))^\top DS(S(y)) v(y) \dd y\\[8pt]
\;&=\int_U \left(\nabla m(y)- m(y) v(y)\right)\cdot\nabla\ov \varphi(y) \dd y\\[8pt]
\; & =0,
\end{array}
\end{equation}
where we have used the fact that both $\varphi$ and $\ov\varphi$ (restricted to $U$) are admissible test functions in \eqref{principalecweakform}, a change of variable in the integrals and the properties that we have shown for $S$.

The regularity of $\hat v$ and   \cite[Corollary 1.7.2]{MR3443169} imply that $m(x)>0$. The result follows.
\end{proof}

\end{document}